\documentclass{amsart}

\usepackage{graphicx}
\usepackage{amscd,amsmath,amsopn,amssymb, amsthm,multicol}  
\usepackage{hyperref}
\usepackage[all]{xy}
\usepackage{amscd}
\usepackage{lscape}
\usepackage{slashed}
\usepackage{graphicx}
\usepackage{lscape}
 \usepackage{cancel}
  
\usepackage{mathrsfs}
\DeclareMathAlphabet{\mathscrbf}{OMS}{mdugm}{b}{n}

\newtheorem{theorem}{Theorem}[section]
\newtheorem{lemma}[theorem]{Lemma}
\newtheorem{prop}[theorem]{Proposition}
\newtheorem{corol}[theorem]{Corollary}

\theoremstyle{definition}
\newtheorem{definition}[theorem]{Definition}
\newtheorem{example}[theorem]{Example}
 
\theoremstyle{remark}

\numberwithin{equation}{section}

 \input ulem.sty
 
\DeclareMathOperator{\Id}{Id}
\DeclareMathOperator{\tr}{tr}

\DeclareMathOperator{\Sym}{Sym}

\DeclareMathOperator{\Cl}{C\ell}
\DeclareMathOperator{\Ric}{Ric}
\DeclareMathOperator{\Sca}{Scal}

\DeclareMathOperator{\ke}{Ker}

\DeclareMathOperator{\Span}{span}
\DeclareMathOperator{\Scho}{Sch}
\DeclareMathOperator{\Spin}{Spin}
\DeclareMathOperator{\rnk}{rnk}

\newcommand{\fr}{\mathfrak}
\newcommand{\al}{\alpha}
\newcommand{\be}{\beta}

\newcommand{\bb}{\mathbb}
\newcommand{\cal}{\mathcal}

\DeclareMathOperator{\SO}{SO}
\DeclareMathOperator{\Sp}{Sp}
 \DeclareMathOperator{\SU}{SU}
  \DeclareMathOperator{\Spec}{Spec}
\DeclareMathOperator{\U}{U}
\DeclareMathOperator{\G}{G}

\DeclareMathOperator{\Ss}{S}
\DeclareMathOperator{\Ed}{End}
\DeclareMathOperator{\Gl}{GL}
\newcommand{\thickline}{\noalign{\hrule height 1pt}}

\begin{document}   

 \title
{Killing and twistor spinors   with torsion}
\author{Ioannis Chrysikos}
 \address{Department of Mathematics and Statistics, Masaryk University, Brno  611 37, Czech Republic}
 \email{chrysikosi@math.muni.cz}




 

 \begin{abstract}
We study twistor spinors (with torsion) on Riemannian spin manifolds $(M^{n}, g, T)$ carrying metric connections with totally skew-symmetric torsion.  We consider the characteristic connection $\nabla^{c}=\nabla^{g}+\frac{1}{2}T$ and under the condition  $\nabla^{c}T=0$, we show  that the twistor equation with torsion w.r.t. the family $\nabla^{s}=\nabla^{g}+2sT$ can be viewed as a parallelism condition under a suitable connection on the  bundle $\Sigma\oplus\Sigma$, where $\Sigma$ is the associated spinor bundle.  Consequently, we prove  that   a twistor spinor with torsion has isolated zero points.  Next we study a special class of twistor spinors with torsion, namely these which are $T$-eigenspinors  and parallel under the characteristic connection; we show  that  the existence of such a spinor   for some $s\neq  1/4$  implies that   $(M^{n}, g, T)$ is both Einstein and $\nabla^{c}$-Einstein, in particular  the equation $\Ric^{s}=\frac{\Sca^{s}}{n}g$   holds for any $s\in\bb{R}$.  In fact,  for  $\nabla^{c}$-parallel  spinors  we provide a correspondence between  the   Killing spinor  equation    with torsion and the Riemannian Killing spinor  equation.    This allows us to describe 1-parameter families of non-trivial Killing spinors with torsion   on  nearly K\"ahler manifolds  and   nearly parallel $\G_2$-manifolds, in dimensions 6 and 7, respectively, but also on  the 3-dimensional sphere $\Ss^{3}$. We finally present applications related to the universal and twistorial eigenvalue estimate  of the square of the cubic Dirac operator.
  
       \medskip
 \noindent 2000 {\it Mathematics Subject Classification.}    53C25-28, 53C30, 58J60.
 
 \noindent {\it Keywords}:  characteristic connection,  parallel spinor,   twistor spinor,  Killing spinor with torsion,  $\nabla$-Einstein structure, cubic  Dirac operator
   \end{abstract}
\maketitle  
 
\section*{Introduction}
Consider a  connected  Riemannian spin manifold $(M^{n}, g)$  carrying a family of  metric connections $\nabla^{s}$ with   (totally) skew-symmetric torsion, say $4sT$ for some non-trivial 3-form $T\in \bigwedge^{3}M$ and $s\in\bb{R}$.  We use the same symbol for the lift of $\nabla$ on the spinor bundle $\Sigma\to M$  and  denote by  $X\cdot \varphi:=\mu(X\otimes\varphi)$   the Clifford multiplication $\mu : TM\otimes \Sigma\to \Sigma$  at the bundle level. If there exist some $s\neq 0$ and $\zeta\neq 0$ such that   
\[
\nabla^{s}_{X}\varphi=\zeta X\cdot\varphi, \quad (\ast)
\]
for  any vector field $X$ on $M$, then  the spinor field $\varphi\in\Gamma(\Sigma)$ is called a {\it Killing spinor with  torsion}  (KsT in short). For $s=0$  (zero torsion), this equation reduces to the   Riemannian Killing spinor equation with its strong geometric  consequences, see for example \cite{Baum, Fr}.   Similarly, a {\it twistor spinor  with torsion} (TsT in short), is a  section $\varphi\in\Gamma(\Sigma)$ satisfying the twistor equation with respect to $\nabla^{s}$ for some $s\neq 0$, i.e.
  \[
\nabla^{s}_{X}\varphi+\frac{1}{n}X\cdot D^{s}\varphi=0, \quad (\ast\ast)
\]
  where $D^{s}=\mu\circ\nabla^{s}$ is the induced Dirac operator.     The interest in   TsT  and in some special cases  KsT   (e.g. in dimension 6, see \cite{ABK}, but also in dimension 7  as we will show in the present work),  is due to the fact that 
 they realize    the equality case for eigenvalue estimates of     the  cubic Dirac operator  $\slashed{D}=D^{g}+\frac{1}{4}T$, i.e. the Dirac operator   associated to the connection with torsion $T/3$.  These   estimates    occur both in the presence of parallel torsion.  The so-called {\it universal estimate} $\be_{\rm univ}$ is based on  a  generalized formula  of  Schr\"ondinger-Lichnerowicz type, see \cite{Bismut, Kos, Dalakov,  FrIv, Agr03, AF}  for a description of the long history related with the square of  $\slashed{D}$.  The {\it  twistorial estimate} $\be_{\rm tw}$ is  more recent and relies on a  method using   Penrose's twistor operator associated to $\nabla^{s}$, see   \cite{ABK, Julia}. The equality case for this estimate takes place when the Riemannian scalar curvature is constant and  the spinor is a twistor spinor  with torsion (TsT) for a specific parameter $s$ (depending on the dimension $n$ of $M$).  Hence, it gives rise to a technique available for the construction of  non-trivial TsT.        It is an interesting question to check if these twistor spinors with torsion  are also some kind of Killing spinors and what the geometric inclusions are when the two estimates coincide, if any.

 A customary trick to attack  these problems is the assumption of parallel torsion.  To be more precise, throughout this paper we shall be interested in connected Riemannian spin manifolds     $(M^{n}, g)$   endowed with a non-integrable $G$-structure  $(G\subsetneq \SO_{n})$ and its {\it characteristic connection} $\nabla^{c}$.  This, if it exists, is a  $G$-invariant metric connection with non-trivial skew-torsion $T$ (unique in general), preserving the tensor fields defining the $G$-structure (we refer to   \cite{Srni} for a detailed exposition). Since for $s=1/4$ the family $\nabla^{s}$ has torsion $T$, from now on we set  $\nabla^{1/4}=\nabla^{c}$.      Notice that our requirement    $\nabla^{c}T=0$  is not a very restrictive condition (see \cite{FriedG2, Srni} for several examples of  $G$-structures satisfying this setting). On the other hand, it implies the vanishing of the co-differential   $\delta T=0$ and ensures the compatibility of the actions of $\nabla^{c}$ and $T$ on $\Sigma$.   Hence, the Ricci tensors  $\Ric^{c}$ and more general $\Ric^{s}$, remain   symmetric and the  same time  the  spinor bundle $\Sigma$ decomposes into a direct sum of  eigenbundles $\Sigma_{\gamma}$ corresponding to  the $T$-action.
 
   In this note, we first compute the Ricci tensor and scalar curvature of a manifold $(M^{n}, g, T)$ carrying   twistor spinors  with respect to the family $\nabla^{s}$, under the assumption $\nabla^{c}T=0$.  Then, we  prove that the twistor equation $P^{s}(\varphi)=0$ can be viewed as  a  parallelism condition  with respect to a suitable connection  $\nabla^{s, E}$ on the vector bundle $\Sigma\oplus\Sigma$ (Theorem \ref{labro}). In particular, a section $\varphi\in\Gamma(\Sigma)$ is a  twistor spinor with torsion if and only if the section $\varphi\oplus D^{s}(\varphi)$ of $\Sigma\oplus\Sigma$ is $\nabla^{s, E}$-parallel. In this way we generalize the Riemannian case (compare with \cite[Thm.~4, p.~25]{Baum}, for instance) and we show that   the zeros of such spinors are isolated. 
 
   Next, we take advantage of the  splitting $\Sigma=\oplus_{\gamma}\Sigma_{\gamma}$ and   examine a special class of TsT, namely  these which are parallel under the characteristic connection $\nabla^{c}$.  Notice that the existence of a $\nabla^{c}$-parallel spinor is already  a   strong condition which imposes  restrictions on the  holomomy group of $\nabla^{c}$, see for example \cite{Dalakov, FrIv, FriedG2}.  
 Here, for a compact triple $(M^{n}, g, T)$ of positive scalar  curvature, we  present first a criterion which allows us to decide when a $\nabla^{c}$-parallel spinor   lying inside $\Sigma_{\gamma}$ for some $\gamma\neq 0$, is a real Killing spinor (Proposition \ref{klik}).  For $3< n\leq 8$ we see  that this is   the limiting case of the corresponding criterion related to its  existence \cite[Lem.~4.1]{ABK}.   Furthermore, under certain   assumptions  we  identify these  classes of spinors (Theorem   \ref{newFI}) and next extend this correspondence to  KsT and TsT as well.  In particular, for  a  $\nabla^{c}$-parallel spinor  $\varphi$  we show  that   the Riemannian Killing spinor equation $\nabla^{g}_{X}\varphi=\kappa X\cdot\varphi$ with $\kappa:=3\gamma/4n$ for some   $\gamma\neq 0$ is equivalent to the KsT equation $(\ast)$ for some (and thus any) $s\neq 0, \frac{1}{4}$ with  Killing number   $\zeta:=3(1-4s)\gamma/4n$ and moreover with the twistor equation $(\ast\ast)$ for some (and thus any) $s\neq 1/4$, under  the additional condition   $\varphi\in\Sigma_{\gamma}$ (Theorem \ref{general1}). 
 Using these results and combining with \cite[Thm.~3.4]{FrIv}  we  finally conclude  that on a compact Riemannian manifold of constant scalar curvature given by $\Sca^{g}=\frac{9(n-1)\gamma^{2}}{4n}$ for some non-zero $T$-eigenvalue  $0\neq \gamma\in\Spec(T)$, the following   classes of spinors, if existent, coincide  
   \[
 \ke(\nabla^{c})\cong \bigoplus_{\gamma}\Big[\Gamma(\Sigma_{\gamma})\cap\cal{K}(M, g)_{\frac{3\gamma}{4n}}\Big]\cong\bigoplus_{\gamma} \Big[\Gamma(\Sigma_{\gamma})\cap\cal{K}^{s}(M, g)_{\frac{3(1-4s)\gamma}{4n}}\Big]\cong \bigoplus_{\gamma} \Big[\ke(P^{s}\big|_{\Sigma_{\gamma}})\cap\ke\big(D^{c}\big)\Big],
  \]
under the assumption that  $\nabla^{c}T=0$ and that the symmetric endomorphism  $dT+\frac{1}{2}\Big[\frac{9(n-1)}{4n}\gamma^{2}-\frac{3}{2}\|T\|^{2}\Big]$ acts on $\Sigma$ with non-negative eigenvalues (see Theorem  \ref{newFI}, Corollary \ref{addd}).     
   
   In the following we examine the geometric constrains that imposes the existence of such spinors.   Due to the previous correspondence, it is obvious that   a triple $(M^{n}, g, T)$ $(n\geq 3)$ carrying a non-trivial $\nabla^{c}$-parallel spinor field  $\varphi\in\Gamma(\Sigma)$  which is a Killing spinor with torsion (KsT) with respect to the family $\nabla^{s}=\nabla^{g}+2sT$ with Killing number $\zeta=3\gamma(1-4s)/4n\neq 0$, must  be Einstein
  \[
    \Ric^{g}= \frac{9(n-1)\gamma^{2}}{4n^{2}}\Id_{TM}.  \quad \text{\sc (I)}
    \] 
  Moreover, the equation  $T\cdot\varphi=\gamma\varphi$ needs to hold.  However, one can say much more; we prove that $(M^{n}, g, T)$ is  also $\nabla^{c}$-Einstein    (in fact, for $n=3$ such an manifold must be $\Ric^{c}$-flat, although never $\Ric^{g}$-flat, see Proposition \ref{bravo})
   \[
 \Ric^{c} =\frac{3(n-3)\gamma^{2}}{n^{2}}\Id_{TM}. \quad \text{\sc (II)}
\]
  To do this, we  follow a spinorial approach to  $\nabla^{c}$-Einstein manifolds carrying a $\nabla^{c}$-parallel spinor $\varphi\in\Sigma_{\gamma}$, which is  available for any triple $(M^{n}, g, T)$ admitting parallel spinors  $\varphi\in\Sigma_{\gamma}$    with respect to a metric connection $\nabla$ with {\it parallel} skew-torsion $T$ (see for example  \cite{FriedG2}).  In  our case,  this technique    allows us  to deduce  that $(M^{n}, g, T)$ is $\nabla^{c}$-Einstein and   then we use this result to provide an alternative proof of the original Einstein condition, independent of the fact  that $\varphi$  must be   a real Killing spinor.     In a sense, this is   the opposite of the  way that  $\nabla^{c}$-Einstein structures have been    traditionally examined, especially   in dimensions 6 and 7 (see e.g.  \cite[Prop.~10.4]{FrIv}),  but also in  more general cases, e.g.  naturally reductive spaces (see \cite{Chrysk}). Taking advantage of our study on twistor spinors   we finally deduce that   a triple  $(M^{n}, g, T)$ endowed with  a $\nabla^{c}$-parallel KsT with $\zeta=3\gamma(1-4s)/4n\neq 0$,  must be $\nabla^{s}$-Einstein (with non-parallel torsion) for any $s\neq 0, 1/4$.  In particular,  the equation $\Ric^{s}=\frac{\Sca^{s}}{n}g$ is satisfied for any $s$, with the values $s=0, 1/4$ being the special values described above (Proposition \ref{bravo}).  
  Hence, {\it the existence of  $\nabla^{c}$-parallel real Killing spinor with Killing number $\kappa=3\gamma/4n$ for some $0\neq \gamma\in\Spec(T)$, or equivalent  the existence of a $\nabla^{c}$-parallel KsT for some $s\neq 0, 1/4$ with $\zeta=3\gamma(1-4s)/4n$,  implies that }  \[
  \Ric^{s}=\frac{\Sca^{s}}{n}g, \quad \forall \ s\in\bb{R}. \quad \text{\sc (III)}
  \]

In this point, one should  emphasise  that our   results  have been described with respect to the same Riemannian metric (without any deformation).  In particular, Theorem \ref{newFI},  Theorem \ref{general1} and Proposition \ref{bravo}  highlight  this case and allow us to    provide new examples (related with the existence of KsT).  There are two  representative classes of non-integrable $G$-structures  carrying this special kind of KsT;  nearly parallel $\G_2$-manifolds in dimension 7 and    nearly K\"ahler manifolds in dimension 6.   
Another remarkable example is the round 3-sphere $\Ss^{3}$, where $\nabla^{c}$  is not unique and coincides with    the flat $\pm 1$-connections of Cartan-Schouten.  
However, we point out that Einstein-Sasakian manifolds  in any odd dimension $\geq 5$  cannot be  candidates of   Theorem \ref{general1} for example,    since according to  \cite[Rem.~2.26]{AFer} such a manifold is never $\nabla^{c}$-Einstein.  Indeed, the integrability conditions ({\sc{I}}) and ({\sc{II}})  are already very strong and  provide a recipe to describe   several special structures endowed with their characteristic connection  that fail to carry $\nabla^{c}$-parallel Killing spinors with torsion with respect to  $\nabla^{s}=\nabla^{g}+2sT$.  To avoid confusions, we   mention that any Einstein-Sasaki manifold $M^{n}$ $(n\geq 5)$ carries KsT and these occur  of the real Killing spinors after applying the Tanno deformation on the Einstein-Sasaki metric, see   \cite[Ex.~3.16]{AHol} or \cite[Ex.~5.1, 5.2] {ABK} and for the  full picture   the Phd thesis  \cite{Julia}. However, for these spinors the parallelism equation and the KsT equation do not hold anymore with respect to the same metric. This is the  main difference with the KsT lying at the heart of the present work; they are always parallel under  the characteristic connection  and both our spinorial equations are verified with  respect to the same metric  for which the equation $\nabla^{c}\varphi=0$ holds. 

  For    nearly parallel $\G_2$-structures, Theorem \ref{general1}  gives rise to a   complete  description of all non-trivial $\nabla^{c}$-parallel KsT that such a manifold  admits (Theorem \ref{nG2}).   The same applies  in dimension 6 for nearly K\"ahler manifolds, with the difference that   the result was known for  $s=5/12$, see \cite[Thm.~6.1]{ABK} or \cite{Julia}.  Here, we    extend this correspondence to any parameter $s\in\bb{R}\backslash\{0, 1/4\}$ (Theorem \ref{nk}).  
    Notice    that the existence of such spinors   on  nearly parallel $\G_2$-manifolds  has been conjectured in \cite{ABK}, and this was part of our motivation.
  We  finally  remark that  these weak holonomy structures are both solutions of the equations for the common sector of type II superstring theory: $\nabla^{c}\varphi=0$, $T\cdot\varphi=\gamma\cdot \varphi$, $\delta(T)=0$ and $\nabla^{c}\Ric^{c}=0$  (see \cite{FrIv, Agr03, AFNP, Srni}).
 The supersymmetries  of the model are interpreted by the  $\nabla^{c}$-parallel spinors.  Our results show that there are models of this kind, which  give rise to  solutions of the Killing spinor equation with torsion and satisfy equation ({\sc{III}}), for any $s\in\bb{R}$.
       
  In the last part of the paper, we  examine   some further applications. 
  We  show for example that the integrability condition related to the existence of general KsT (for $s=(n-1)/4(n-3)$, see  \cite[Thm.~A.2]{ABK}), reduces in the special case of  a $\nabla^{c}$-parallel KsT with $\zeta=3\gamma(1-4s)/4n$  to  an  identity, namely the twistor equation with torsion (Corollary \ref{little}).
  After that, we focus on the inequality $\beta_{\rm tw}(\gamma)\leq \beta_{\rm univ}(\gamma)$.    For $3<n\leq 8$ and for  a spinor $\varphi\in\Sigma_{\gamma}$ satisfying the equation $\nabla^{c}\varphi=0$,    we prove that  the  equation $\beta_{\rm tw}(\gamma)=\beta_{\rm univ}(\gamma)$ is equivalent to say that      $\varphi$ is a real Killing spinor   with Killing number $\kappa=3\gamma/4n$ (Proposition \ref{afterklik}).  As a consequence, in dimension six this special case is exhausted by  nearly K\"ahler manifolds (see   \cite[Ex.~6.1]{ABK}) and  in dimension 7 by nearly parallel $\G_2$-manifolds.

  \smallskip
\noindent {\bf Acknowledgements.} The author  acknowledges   support  by  GA\v{C}R (Czech Science Foundation), post-doctoral grant  no.14-2464P.  
He  warmly  thanks  Ilka Agricola (Marburg) and Thomas Friedrich (Berlin) for   valuable  discussions that  improved this work, as well as,   Ivan Minchev (Brno) and Arman Taghavi-Chabert (Brno)  for  very useful   comments. He also  acknowledges FB12 at Philipps-Universit\"at Marburg for its hospitality during a research  stay in April 2015.

   	\section{Preliminaries}
	Let us consider a connected Riemannian spin manifold $(M^{n}, g, T)$  $(n=\dim M\geq 3)$,  carrying a non-trivial 3-form $T\in\bigwedge^{3}M$, as in introduction.\footnote{In this paper all the manifolds, tensor fields and other geometric objects under consideration, are  assumed to be smooth.}   Recall that for some $s\in\bb{R}$,  $\nabla^{s}$ is  the metric connection  with  skew-torsion $4sT$. This is defined by
	\[
 g(\nabla^{s}_{X}Y, Z)=g(\nabla^{g}_{X}Y, Z)+2sT(X, Y, Z) 
\]
 and    joins the characteristic connection  $\nabla^{1/4}\equiv \nabla^{c}$  with the Levi-Civita connection $\nabla^{0}\equiv\nabla^{g}$.  A Riemannian  manifold endowed with a metric connection $\nabla\equiv\nabla^{c}$ with totally skew-symmetric torsion $T$, is usually called Riemannian manifold with torsion and the geometry Riemann-Cartan geometry  (see for example \cite{Dalakov}). Here we will not use this notation and whenever we  refer to a triple $(M^{n}, g, T)$ we shall mean a connected Riemannian spin manifold  endowed with the above family of connections. It is   useful to consider   the 4-form \[\sigma_{T}:=\frac{1}{2}\sum_{i=1}^{n}(e_{i}\lrcorner T)\wedge (e_{i}\lrcorner T)\]   and normalize the length of the 3-form $T$ as $  \|T\|^{2}=\frac{1}{3}\sum_{i\leq i<j\leq n}g\big(T(e_{i}, e_{j}), T(e_{i}, e_{j})\big)$,   where $\{e_1, \ldots, e_n\}$ is an orthonormal frame of $(M^{n}, g)$. The Riemannian scalar curvature $\Sca^{g}$ and the scalar curvature induced by $\nabla^{s}$ are connected by the rule $\Sca^{s}=\Sca^{g}-24s^{2}\|T\|^{2}$. Moreover, inside the Clifford algebra we have that $2\sigma_{T}=\|T\|^{2}-T^{2}$, for   details and proofs see \cite{Dalakov, IvPap, FrIv, Agr03, AF, Srni, ABK}.   The lift of $\nabla^{s}$ to the spinor bundle $\Sigma$ is given by 
\[ \nabla^{s}_{X}\varphi=\nabla^{g}_{X}\varphi+s(X\lrcorner T)\cdot\varphi \]
 for any $X\in\Gamma(TM)$ and $\varphi\in\Gamma(\Sigma)$. After identifying $TM\cong T^{*}M$ via the metric tensor,  the associated Dirac operator $D^{s} : \Gamma(\Sigma)\overset{\nabla^{s}}{\to}\Gamma(TM\otimes \Sigma)\overset{\mu}{\to}\Gamma(\Sigma)$ reads 
\[ D^{s}(\varphi)=\sum_{i}e_{i}\cdot\nabla^{s}_{e_{i}}\varphi=\sum_{i}e_{i}\cdot(\nabla^{g}_{e_{i}}\varphi+s(e_{i}\lrcorner T)\cdot\varphi)=D^{g}(\varphi)+3sT\cdot\varphi,  \]
 where $D^{g}\equiv D^{0}:=\mu\circ\nabla^{0}$ is the Riemannian Dirac operator. For some $s\in\bb{R}\backslash\{0\}$ we shall denote by 
 \[
 \cal{K}^{s}(M, g)_{\zeta}:=\{\varphi\in\Gamma(\Sigma) :  \nabla^{s}_{X}\varphi=\zeta X\cdot\varphi \ \ \forall X\in\Gamma(TM)\}
 \]
  the set of all  {\it Killing spinors with torsion} (KsT) with respect to the family $\nabla^{s}=\nabla^{g}+2sT$   with   Killing number    $\zeta\neq 0$.  Similarly,   $\cal{K}(M, g)_{\kappa}$ will denote the set of Riemannian Killing spinors  with  Killing number $\kappa\neq 0$.  In general,    the  Killing number   $\zeta$ can be   a complex number; however  only solutions with $\zeta\in\bb{R}\backslash\{0\}$ are known (see also \cite{AHol}).   Next, we are mainly interested in real Killing numbers (with torsion or not).  Notice  also that  we do {\it not} view   $\nabla^{s}$-parallel spinor  as a special case of a KsT.
 We finally remark that in \cite[Def.~5.1]{ABK}, the parameter $s$ is ``fixed'', in the sense that it depends on $n=\dim M$, namely $s=(n-1)/4(n-3)$, see \cite{ABK} for more explanations. Here we relax this condition and  follow the definition of \cite{AHol}, which is more general. Let us also recall that
   \begin{definition}\textnormal{(\cite{ABK})}
 A {\it twistor spinor with torsion} (TsT)   is a   spinor field solving  the differential equation $\nabla^{s}_{X}\varphi+(1/n)X\cdot D^{s}(\varphi)=0$, for any $X\in\Gamma(TM)$, i.e. an element in the kernel of the    Penrose or twistor operator  $P^{s}$ associated to $\nabla^{s}$. This is the differential operator defined by the composition 
 \[
 \Gamma(\Sigma)\overset{\nabla^{s}}{\to}\Gamma(TM\otimes \Sigma)\overset{p}{\to}\Gamma(\ker\mu), \quad P^{s}=p\circ\nabla^{s},
 \]
   where $p : TM\otimes\Sigma\to\ker\mu\subset TM\otimes\Sigma$ is the orthogonal projection  onto the kernel of the Clifford multiplication.  Locally one has
     \[
  p(X\otimes\varphi):=X\otimes\varphi+\frac{1}{n}\sum_{i=1}^{n}e_{i}\otimes e_{i}\cdot X\cdot\varphi, \quad P^{s}\varphi:= \sum_{i=1}^{n} e_{i}\otimes\{\nabla^{s}_{e_{i}}\varphi+\frac{1}{n}e_{i}\cdot D^{s}\varphi\}. 
\]
   \end{definition}
  Any KsT with Killing number $\zeta\neq 0$ is  a  $D^{s}$-eigenspinor, i.e. $D^{s}\varphi=-n \zeta  \varphi$ and thus  a special solution of the twistor equation with torsion.  And conversely, any twistor spinor $\varphi\in\ke(P^{s})$ which is the same time a  $D^{s}$-eigenspinor, is also a  KsT.
  It is easy to see that TsT are satisfying the same basic properties with Riemannian twistor spinors \cite[Th.~3.1]{ABK}. In the following we  shall develop  a theory for  TsT, as an analogue of the Riemannian case. 

\section{Twistor spinors with torsion}\label{good2}
 \subsection{Twistor spinors with torsion} From now on, and for the following of this article  we assume that the characteristic connection $\nabla^{c}:=\nabla^{g}+(1/2)T$ satisfies the equation $\nabla^{c}T=0$.  Then, the relations $\delta{T}=0$ and $dT=2\sigma_{T}$ hold and the length $\|T\|^{2}$ is constant, see \cite{AF, Srni}.   
Moreover,  the curvature tensor $R^{s}$ is symmetric $R^{s}(X, Y, Z, W)=R^{s}(Z, W, X, Y)$ and the same holds for the Ricci tensor $\Ric^{s}(X, Y)=\sum_{i}R^{s}(X, e_{i}, e_{i}, Y)$ (since $\delta^{s}T=0$, see \cite[Thm.~B.1]{ABK}).  In fact, one can write $\Ric^{s}(X, Y)=\Ric^{g}(X, Y)-4s^{2}S(X, Y)$, where $S$ is the symmetric tensor given by $S(X, Y):=\sum_{i}g(T(X, e_{i}), T(Y, e_{i}))$. This easily occurs by combining for example the formulas \cite[p.~740]{AFer}  \begin{eqnarray*}
 R^{s}(X, Y, Z, W)&=&R^{g}(X, Y, Z, W)+4s^{2}\big[g(T(X, Y), T(Z, W))+\sigma_{T}(X, Y, Z, W)\big]\\
 &&+2s\big[\nabla_{X}^{s}T(Y, Z, W)-\nabla^{s}_{Y}T(X, Z, W)\big].
 \end{eqnarray*}
and   $\nabla^{s}_{X}T(Y, Z, V)=\frac{4s-1}{2}\sigma_{T}(Y, Z, V, X)$    \cite[Thm.~B.1]{ABK}. 
 Passing now to the spinor bundle $\Sigma$, the curvature tensor $R^{s}$ associated to the lift of $\nabla^{s}$ is defined  by $R^{s}(X, Y)\varphi=\nabla^{s}_{X}\nabla^{s}_{Y}\varphi-\nabla^{s}_{Y}\nabla^{s}_{X}\varphi-\nabla^{s}_{[X, Y]}\varphi$. In her  Phd thesis, J.~Becker-Bender  proved that the Ricci endomorphism $\Ric^{s}(X):=\sum_{i}\Ric^{s}(X, e_{i}) e_{i}\cdot\varphi$ is related with $R^{s}$ and $\sigma_{T}$ as follows  (see also \cite[Thm.~A.1]{ABK} for $s=1/4$):\begin{lemma}\textnormal{(\cite[Lem.~1.13]{Julia})} \label{JUL1}
   \[
 \sum_{i}e_{i}\cdot R^{s}(X, e_{i})\varphi=-\frac{1}{2}\Ric^{s}(X)\cdot\varphi+s(3-4s)(X\lrcorner \sigma_{T})\cdot\varphi, \quad \forall \ \varphi\in\Gamma(\Sigma), \ X\in\Gamma(TM).
 \]
 \end{lemma}
We present now the Ricci curvature  associated to  the family $\nabla^{s}$, on a manifold $(M^{n}, g, T)$  carrying a twistor spinor  with torsion w.r.t. $\nabla^{s}$, i.e. $\nabla^{s}_{X}\varphi=-\frac{1}{n}X\cdot D^{s}(\varphi)$, under the condition $\nabla^{c}T=0$.  
 \begin{lemma}\label{good}
 For any twistor spinor $\varphi\in\ke(P^{s})$ and for any vector field $X$ the following relations hold:
 \begin{eqnarray*}
 -\frac{1}{2}\Ric^{s}(X)\cdot\varphi&=&-\frac{8s}{n}(X\lrcorner T)\cdot D^{s}(\varphi)+\frac{n-2}{n}\nabla^{s}_{X}\big(D^{s}(\varphi)\big)-\frac{1}{n}X\cdot (D^{s})^{2}(\varphi) -s(3-4s)(X\lrcorner \sigma_{T})\cdot\varphi.\\
 \frac{1}{2}\Sca^{s}\varphi&=&-\frac{24s}{n}T\cdot D^{s}(\varphi)+\frac{2(n-1)}{n}(D^{s})^{2}(\varphi)-4s(3-4s)\sigma_{T}\cdot\varphi.
 \end{eqnarray*}
 \end{lemma}
 \begin{proof}
 Consider a twistor spinor $\varphi\in\ke(P^{s})$ for some $s\in\bb{R}$. For some $X, Y\in\Gamma(TM)$  it is 
\begin{eqnarray*}
 \nabla^{s}_{X}\nabla^{s}_{Y}\varphi&=&-\frac{1}{n}(\nabla^{s}_{X}Y)\cdot D^{s}(\varphi)-\frac{1}{n}Y\cdot \nabla^{s}_{X}\big(D^{s}(\varphi)\big),\\
   \nabla^{s}_{Y}\nabla^{s}_{X}\varphi&=&-\frac{1}{n}(\nabla^{s}_{Y}X)\cdot D^{s}(\varphi)-\frac{1}{n}X\cdot \nabla^{s}_{Y}\big(D^{s}(\varphi)\big).
\end{eqnarray*}
Hence,  for the curvature tensor  on $\Sigma$  we compute
\[
 R^{s}(X, Y)\varphi =-\frac{1}{n}T^{s}(X, Y)\cdot D^{s}(\varphi)-\frac{1}{n}\big[Y\cdot \nabla^{s}_{X}\big(D^{s}(\varphi)\big)-X\cdot \nabla^{s}_{Y}\big(D^{s}(\varphi)\big)\big],
\]
 where $T^{s}(X, Y):=\nabla^{s}_{X}Y-\nabla^{s}_{Y}X-[X, Y]=4sT(X, Y)$ is the torsion of the family $\nabla^{s}$.  Let $\{e_{1}, \ldots, e_{n}\}$ a local orthonormal frame of $TM$. Then, multiplying by $e_{i}$ and summing we conclude that
 \begin{eqnarray*}
 \sum_{i}e_{i}\cdot R^{s}(X, e_{i})\varphi&=&-\frac{4s}{n}\sum_{i}e_{i}\cdot T(X, e_{i})\cdot D^{s}(\varphi)-\frac{1}{n}\sum_{i}e_{i}\cdot e_{i} \cdot \nabla^{s}_{X}\big(D^{s}(\varphi)\big)\\
 &&+\frac{1}{n}\sum_{i}e_{i}\cdot X\cdot \nabla^{s}_{e_{i}}\big(D^{s}(\varphi)\big)\\
 &=&-\frac{8s}{n}(X\lrcorner T)\cdot D^{s}(\varphi)+\nabla^{s}_{X}\big(D^{s}(\varphi)\big)+\frac{1}{n}\sum_{i}e_{i}\cdot X\cdot \nabla^{s}_{e_{i}}\big(D^{s}(\varphi)\big),
 \end{eqnarray*}
 where   the relation   $\sum_{i}e_{i}\cdot T(X, e_{i})=2(X\lrcorner T)$ was used (see \cite[p.~325]{ABK}). 
 Recalling that $e_{i}\cdot X+X\cdot e_{i}=-2g(e_{i}, X)\Id_{\Sigma}$, we also get
 $\frac{1}{n}\sum_{i}e_{i}\cdot X\cdot \nabla^{s}_{e_{i}}\big(D^{s}(\varphi)\big)=-\frac{1}{n}X\cdot (D^{s})^{2}(\varphi)-\frac{2}{n}\nabla^{s}_{X}\big(D^{s}(\varphi)\big)$.  Therefore, for     any   $\varphi\in\ke(P^{s})$ the following holds:
 \begin{eqnarray*}
\sum_{i}e_{i}\cdot R^{s}(X, e_{i})\varphi&=&-\frac{8s}{n}(X\lrcorner T)\cdot D^{s}(\varphi)+\frac{n-2}{n}\nabla^{s}_{X}\big(D^{s}(\varphi)\big)-\frac{1}{n}X\cdot (D^{s})^{2}(\varphi).
 \end{eqnarray*}
 Now, our first assertion  is an immediate consequence of  Lemma \ref{JUL1}.
 We proceed with the scalar curvature. Since $\Ric^{s}$ is symmetric, we have that
\[
\sum_{i}e_{i}\cdot \Ric^{s}(e_{i})\cdot\varphi=\sum_{i, j}\Ric^{s}(e_{i}, e_{j})\cdot e_{i}\cdot e_{j}\cdot\varphi=-\sum_{i}\Ric^{s}(e_{i}, e_{i})\cdot\varphi=-\Sca^{c}\cdot\varphi.
\]
But  then,   combing  with the expression of $\Ric^{s}$ and  observing that $\sum_{i}e_{i}\cdot (e_{i}\lrcorner T)=3T$, $\sum_{i}e_{i}\cdot (e_{i}\lrcorner \sigma_{T})=4\sigma_{T}$ and $\sum_{i}e_{i}\cdot \nabla^{s}_{e_{i}}\big(D^{s}(\varphi)\big)=(D^{s})^{2}(\varphi)$, we easily finish the proof.  Notice that the given formula of $\Sca^{s}$,    encodes also  the action of the square of the Dirac operator $D^{s}$ on twistor spinors. 
Indeed, in our case $\nabla^{c}T=0$ it is well-known that the following holds (see  \cite[Thm.~6.1]{AF} or \cite[Thm.~2.1]{ABK})
\[
(D^{s})^{2}(\varphi)=\Delta^{s}(\varphi)+ s(3-4s) dT\cdot\varphi-4s\cal{D}^{s}(\varphi)+\frac{1}{4}\Sca^{c}\varphi, 
\]
 where $\cal{D}^{s}:=\sum_{i}(e_{i}\lrcorner T)\cdot\nabla^{s}_{e_{i}}\varphi$. But the action of the differential operators  $\Delta^{s}:=(\nabla^{s})^{*}\nabla^{s}$ and $\cal{D}^{s}$ on some $\varphi\in\ke(P^{s})$, is given by 
 \[
 \cal{D}^{s}(\varphi)=\sum_{i}(e_{i}\lrcorner T)\cdot\nabla^{s}_{e_{i}}\varphi=-\frac{1}{n}\sum_{i}(e_{i}\lrcorner T)\cdot e_{i}\cdot D^{s}(\varphi)=-\frac{3}{n}T\cdot D^{s}(\varphi)
 \]
 and  $\Delta^{s}(\varphi)=\frac{1}{n}(D^{s})^{2}(\varphi)$, respectively. Hence, for a twistor spinor with torsion we obtain that
 \begin{eqnarray*}
 (D^{s})^{2}(\varphi)&=&\frac{n}{n-1}\big[s(3-4s) dT\cdot\varphi-4s\cal{D}^{s}(\varphi)+\frac{1}{4}\Sca^{s}\varphi\Big]\\
 &=&\frac{n}{n-1}\big[s(3-4s) dT\cdot\varphi+\frac{12s}{n}T\cdot D^{s}(\varphi)+\frac{1}{4}\Sca^{s}\varphi\Big], \quad (\natural)
  \end{eqnarray*}
 which is equivalent with the given expression of $\Sca^{s}$ (observing that $dT=2\sigma_{T}$).  
 \end{proof}

\noindent{\bf Remarks:}
 For $s=0$, i.e. for the Riemannian connection (zero torsion $T\equiv 0$),  Lemma \ref{good}  reduces to    a basic result of   A.~Lichnerowicz \cite{Lic} about   Riemannian twistor spinors (see also  \cite[pp.~23-24]{Baum},  \cite[pp.~122-123]{Fr}   or \cite[p.~134]{Ginoux}):
 \[
 \nabla^{g}_{X}\big(D^{g}(\varphi)\big)=\frac{n}{2(n-2)}\Big[-\Ric^{g}(X)\cdot\varphi+\frac{\Sca^{g}}{2(n-1)}X\cdot\varphi\Big]=\frac{n}{2}\Scho^{g}(X)\cdot\varphi, \quad (D^{g})^{2}(\varphi)=\frac{n\Sca^{g}}{4(n-1)}\varphi,
 \]
 where $\Scho^{g}(X):=\frac{1}{n-2}\big[-\Ric^{g}(X)\cdot\varphi+\frac{\Sca^{g}}{2(n-1)}X\big]$ is the endomorphism induced by the Schouten tensor of $\nabla^{g}$. For the family $\nabla^{s}$, using $(\natural)$ and  Lemma \ref{good}, we conclude that any element in $\ke(P^{s})$ satisfies the more general formula
 \begin{eqnarray}
\nabla^{s}_{X}\big(D^{s}(\varphi)\big) 
 &=&\frac{n}{2}\Scho^{s}(X)\cdot\varphi+\frac{sn}{(n-1)(n-2)}\Big[\Big(\frac{8(n-1)}{n}(X\lrcorner T)+\frac{12}{n}X\cdot T\Big)\cdot D^{s}(\varphi)\nonumber\\
 && +(3-4s)\big(X\cdot dT+(n-1)(X\lrcorner \sigma_{T})\big)\cdot\varphi\Big], \label{LIC}
 \end{eqnarray}
where here, $\Scho^{s}$ is the Schouten endomorphism associated to $\nabla^{s}$.   Notice that the formula for $(D^{g})^{2}$ immediately appears by $(\natural)$ for $s=0$, which is its analogue for connections with parallel skew-torsion.

  In the case that $\varphi\in\Gamma(\Sigma)$ is a Killing spinor with torsion,  Lemma \ref{good} applies and gives rise to  the  Ricci tensor $\Ric^{s}$ and the scalar curvature $\Sca^{s}$   of a Riemannian manifold carrying a Killing spinor with torsion with respect to $\nabla^{s}$ (as in the Riemannian case, too).  Hence, through a  method based on twistor spinors, one can now verify  formulas that are already known by \cite[Lem.~1.14]{Julia}.  We cite them in a corollary   below. 
 
   \begin{corol}\label{lemkst}\textnormal{(\cite[Lem.~1.14]{Julia})}
 Given a Killing spinor with torsion $\varphi\in\Gamma(\Sigma)$ w.r.t. the family $\nabla^{s}=\nabla^{g}+2sT$ and with Killing number $\zeta\in\bb{R}\backslash\{0\}$, the following hold:
 \begin{eqnarray*}
 \Ric^{s}(X)\cdot\varphi&=&4(n-1)\zeta^{2}X\cdot\varphi-16s\zeta(X\lrcorner T)\cdot\varphi+2s(3-4s)(X\lrcorner \sigma_{T})\cdot\varphi,\\
 \Sca^{s}\varphi&=&4n(n-1)\zeta^{2}\varphi+48s\zeta T\cdot\varphi-8s(3-4s)\sigma_{T}\cdot\varphi.
 \end{eqnarray*}
 \end{corol}
 \begin{proof}
It is  an immediate consequence of Lemma \ref{good}. For a direct proof we refer to  \cite{Julia}.
 \end{proof}
Now, in a similar way with the Riemannian case, equation (\ref{LIC}) tell us that the twistor equation with torsion can be viewed as a parallelism equation with respect to a suitable covariant derivative on the bundle $E:=\Sigma\oplus\Sigma$. But let us explain this generalisation in full details.
\begin{lemma}
Consider the mapping $\nabla^{s, E} : \Gamma(E)\to \Gamma(T^{*}M\otimes E)$ given by
\begin{eqnarray*}
\nabla^{s, E}_{X}(\varphi_{1}\oplus\varphi_{2})&=&\Big(\nabla^{s}_{X}\varphi_{1}+\frac{1}{n}X\cdot\varphi_{2}\Big) \oplus \Big(-\frac{n}{2}\Scho^{s}(X)\cdot\varphi_{1}-\frac{sn(3-4s)}{(n-1)(n-2)}\big[X\cdot dT+(n-1)(X\lrcorner \sigma_{T})\big]\cdot\varphi_{1}\\
&&-\frac{sn}{(n-1)(n-2)}\big[\frac{8(n-1)}{n}(X\lrcorner T)+\frac{12}{n}X\cdot T\big]\cdot\varphi_{2}+\nabla^{s}_{X}\varphi_{2}\Big),
\end{eqnarray*}
for any $X\in\Gamma(TM)$ and $\varphi_{1}\oplus\varphi_{2}\in\Gamma(\Sigma\oplus\Sigma)$.  Then, $\nabla^{s, E}$ defines a covariant derivative on the vector bundle $E=\Sigma\oplus\Sigma$.
\end{lemma}
\begin{proof}
The linearity of $\nabla^{s, E}$ is obvious. We need only to check the   rule 
\[
\nabla^{s, E}(f(\varphi_{1}\oplus\varphi_{2}))=df\otimes(\varphi_{1}\oplus\varphi_{2})+f\nabla^{s, E}(\varphi_{1}\oplus\varphi_{2}),
\]
for some smooth function $f\in C^{\infty}(M; \bb{R})$. But this is  a simple consequence of relation  $\nabla^{s}_{X}(f\varphi_{i})=X(f)\varphi_{i}+f\nabla^{s}_{X}\varphi_{i}=df(X)\otimes\varphi_{i}+f\nabla^{s}_{X}\varphi_{i}$ and the fact that all the other parts are tensors (differential forms) acting by Clifford multiplication. 
\end{proof}
Notice that for $s=0$, the connection $\nabla^{0, E}\equiv\nabla^{E}$ coincides with the connection $\nabla^{E}$   described in \cite[p.~61]{Fconf} (see also \cite[p.~25]{Baum}). We conclude that
\begin{theorem}\label{labro}
Let $(M^{n}, g, T)$ $(n\geq 3)$ a connected Riemannian spin manifold with $\nabla^{c}T=0$. Then, any  twistor spinor with torsion $\varphi\in\ke(P^{s})$ satisfies the equation $\nabla^{s, E}_{X}\big(\varphi\oplus D^{s}(\varphi)\big)=0$. Conversely, if $(\varphi\oplus\psi)\in\Gamma(E)$ is $\nabla^{s, E}$-parallel, then $\varphi$ is a twistor spinor with torsion such that $D^{s}(\varphi)=\psi$. \end{theorem}
\begin{proof}
By definition,
{\small{\[
\nabla^{s, E}_{X} = \begin{pmatrix} 
\nabla^{s}_{X}  & (1/n)X \\\\
-\frac{n}{2}\Scho^{s}(X)-\frac{sn(3-4s)}{(n-1)(n-2)}\big[X\cdot dT+(n-1)(X\lrcorner \sigma_{T})\big] & -\frac{sn}{(n-1)(n-2)}\big[\frac{8(n-1)}{n}(X\lrcorner T)+\frac{12}{n}X\cdot T\big]+\nabla^{s}_{X} 
\end{pmatrix}.
\]}}
 The twistor equation  in combination with  (\ref{LIC}) implies now the result:  $\nabla^{s, E}_{X} \begin{pmatrix} 
\varphi \\
D^{s}(\varphi)
\end{pmatrix}=0$.   The converse occurs due to the relation $\nabla^{s}_{e_{i}}\varphi+\frac{1}{n}e_{i}\cdot\psi=0$, after multiplying with $e_{i}$ and adding, see also  \cite[p.~25]{Baum} or \cite[p.~136]{Ginoux} for $s=0$.
\end{proof}
Thus, a section $\varphi\in\Gamma(\Sigma)$ is a  twistor spinor with torsion if and only if the section $\varphi\oplus D^{s}(\varphi)$ of $\Sigma\oplus\Sigma$ is $\nabla^{s, E}$-parallel. Consequently, any element $\varphi\in\ke(P^{s})$  is defined by its values $\varphi_{p}$, $(D^{s}(\varphi))_{p}$ at some point $p\in M$. Hence
\begin{corol} \label{lab1}
Let $(M^{n}, g, T)$ $(n\geq 3)$ a connected Riemannian spin manifold with $\nabla^{c}T=0$. Then,

\noindent   a) The kernel of the twistor operator is a finite dimensional space, i.e. $\dim_{\bb{C}}\ke(P^{s})\leq 2^{[\frac{n}{2}]+1}$.
 
\noindent b) If  $\varphi$ and $D^{s}(\varphi)$ vanish at some point $p\in M$ and $\varphi\in\ke(P^{s})$, then   $\varphi\equiv 0$.
\end{corol}
\begin{proof}
 Both are  consequences of Theorem \ref{labro}.  Notice that $\rnk_{\bb{C}}(\Sigma)=2^{[\frac{n}{2}]}$ and that  parallel sections on vector bundles over connected manifolds are uniquely determined by their value in a single point.  
\end{proof}
Using Theorem \ref{labro} we     present   the corresponding version of a classical property related with the zeros of a twistor spinor, see \cite{Fconf, Baum}.

\smallskip
\noindent {\bf Remarks:} Locally,    one may  assume that  $(\nabla^{s}X)_{p}=0$ at some point $p\in M$ (or in other words $(\nabla^{s}e_{j})_{p}=0$ for some local orthonormal frame of $TM$). We mention  that this  condition   doesn't mean that   locally the torsion $T^{s}=4sT$ vanishes, see for example \cite[p.~307]{FrIv}. 
   \begin{prop}\label{isolated}
Let $(M^{n}, g, T)$ $(n\geq 3)$ be a connected Riemannian spin manifold with $\nabla^{c}T=0$. Then any  zero point of a twistor spinor  with torsion $0\neq \varphi\in\ke(P^{s})$ is isolated, i.e. the zero-set of $\varphi$ is discrete.
\end{prop}
\begin{proof}
Fix some $p\in M$. We compute the Hessian ${\rm Hess}^{\nabla^{s}}$ of the function $|\varphi|^{2}$ in $p$.  We shall denote by $( \ , \ ):={\rm Re}\langle \ , \ \rangle$ the real inner product on $\Sigma$ induced by the corresponding  Hermitian inner product $\langle \ , \ \rangle$ on  $\Sigma$. For a twistor spinor $\varphi\in\ke(P^{s})$ we compute 
$X(|\varphi|^{2})=2(\nabla^{s}_{X}\varphi, \varphi)=-\frac{2}{n}(X\cdot D^{s}(\varphi), \varphi)$. Thus, in combination with the previous  remark we conclude that locally the following holds: 
\begin{eqnarray*}
{\rm Hess}_{p}^{\nabla^{s}}(|\varphi|^{2})(X, Y)&=&\big(XY(|\varphi|^{2})\big)_{p}=-\frac{2}{n}\Big[X\big(Y\cdot D^{s}(\varphi), \varphi\big)\Big]_{p}\\
&=&-\frac{2}{n}\Big[\big(\nabla^{s}_{X}(Y\cdot D^{s}(\varphi)), \varphi\big)+\big(Y\cdot D^{s}(\varphi), \nabla^{s}_{X}\varphi\big)\Big]_{p}\\
&=&-\frac{2}{n}\Big[\big(Y\cdot \nabla^{s}_{X}(D^{s}(\varphi)), \varphi\big)\Big]_{p}+\frac{2}{n^{2}}\Big[\big(Y\cdot D^{s}(\varphi), X\cdot D^{s}(\varphi)\big)\Big]_{p}.
\end{eqnarray*}
Based on (\ref{LIC}) we  show that the  first term $-\frac{2}{n}\Big[\Big(Y\cdot \nabla^{s}_{X}(D^{s}(\varphi)), \varphi\Big)\Big]_{p}$ vanishes, under the assumption   $\varphi_{p}=0$.
The cancelation  of the part related with the Schouten endomorphism is explained as follows:
\begin{eqnarray*}
- \big(\frac{1}{(n-2)}Y\cdot\big[-\Ric^{s}(X)+\frac{\Sca^{s}}{2(n-1)}X\big]\cdot\varphi, \varphi\big) &=&\frac{1}{n-2}\big(Y\cdot \Ric^{s}(X)\cdot\varphi, \varphi\big)-\frac{\Sca^{s}}{(n-1)(n-2)}\big(Y\cdot X\cdot\varphi, \varphi\big)\\
&=&-\frac{|\varphi|^{2}}{n-2}\Ric^{s}(X, Y)+\frac{|\varphi|^{2}\Sca^{s}}{(n-1)(n-2)}g(X, Y).
\end{eqnarray*}
Here, we use $\Ric^{s}(X):=\sum_{i}\Ric^{s}(X, e_{i}) e_{i}\cdot\varphi$ and the general formula $(X\cdot\varphi, Y\cdot\varphi)=g(X, Y)|\varphi|^{2}$. 
Similarly, for the next term we get
\begin{eqnarray*}
\frac{2s}{(n-1)(n-2)}\big(\big[\frac{8(n-1)}{n}(X\lrcorner T)+\frac{12}{n}X\cdot T\big]\cdot D^{s}(\varphi), Y\cdot \varphi\big),
\end{eqnarray*}  
which its evaluation in $p$ equals to zero (since $Y\cdot\varphi_{p}=0$ and $( \ , \ )={\rm Re}\langle \ , \ \rangle$) and the same is true for the last component. Hence finally, if $\varphi_{p}=0$ we conclude that the Hessian is given by 
\[
{\rm Hess}_{p}^{\nabla^{s}}(|\varphi|^{2})(X, Y)=\frac{2}{n^{2}}\Big[\big(Y\cdot D^{s}(\varphi), X\cdot D^{s}(\varphi)\big)\Big]_{p}=\frac{2}{n^{2}}(|D^{s}(\varphi)|^{2})_{p}g_{p}(X, Y).
\]
The claim now follows as in the classical case \cite[Prop.~2]{Fconf};  if $(D^{s}(\varphi))_{p}\neq 0$, then $p$ is a non-degenerate critical point of $|\varphi|^{2}$ and thus an isolated zero point of $\varphi$. If $(D^{s}(\varphi))_{p}=0$, then $\varphi$ must be trivial by Corollary \ref{lab1}.
\end{proof}

We conclude that   twistor spinors with torsion satisfy the same  structural properties with the original  twistor spinors. This should  not be   a surprising result, since there are already known   examples  of twistor spinors associated to linear connections different than $\nabla^{g}$ which still satisfy these characteristic properties. For instance, for twistor spinors on Weyl manifolds or on manifolds with a foliated structure, i.e. transversal twistor spinors,   results like Theorem \ref{labro}, Corollary \ref{lab1} and Proposition \ref{isolated} are available, see \cite[Thm.~2.7, 2.9]{Buch}, and \cite[pp.~14-15]{Ionescu}, respectively.  
However, one has to stress that the (global) behaviour of the  twistor or Killing spinor equation with torsion is in general different than the behaviour of their Riemannian  analogues, in the sense that the existence of such spinors   imposes  different consequences on the manifold, depending on the geometry and the given $G$-structure (see \cite{ABK, Julia, AHol}).  Next we shall describe a very special class of    twistor or Killing spinors with torsion, namely these which are $\nabla^{c}$-parallel   and  the same time elements of some $T$-eigenspace $\Sigma_{\gamma}$ with $\gamma\neq 0$.  

\section{Spinor fields parallel under the characteristic connection}\label{good1}

\subsection{$\nabla^{c}$-parallel spinors} 


Consider   the cubic Dirac operator  $\slashed{D}=D^{g}+\frac{1}{4}T$.  For the square of $\slashed{D}$ the following formula of   Schr\"odinger-Lichnerowicz type is well-known (apply   \cite[Thm.~2.2]{ABK} for $s=1/4$ and  see also \cite{Dalakov, AF, Cas})
\[
\slashed{D}^{2}=\Delta_{T}+\frac{1}{4}\Sca^{g}-\frac{1}{4}T^{2}+\frac{1}{8}\|T\|^{2},
\]
while for the square of $D^{c}$ it holds that (see \cite[Thm.~3.1]{FrIv})
\[
 (D^{c})^{2}(\varphi)=\Delta_{T}(\varphi)+\frac{1}{2}dT\cdot\varphi-\sum_{i}(e_{i}\lrcorner T)\cdot \nabla^{c}_{e_{i}}\varphi+\frac{1}{4}\Sca^{c}\cdot\varphi.
 \]
In both cases, $\Delta_{T}:=(\nabla^{c})^{*}\nabla^{c}:=-\sum_{i}\nabla^{c}_{e_{i}}\nabla^{c}_{e_{i}}+\nabla^{c}_{\nabla^{g}_{e_{i}}e_{i}}$ denotes the spinor Laplacian associated to the connection $\nabla^{c}$.  Because $\nabla^{c}T=0$, the square   $\slashed{D}^{2}$ commutes with the  symmetric endomorphism $T$.  In particular, $T$ acts on spinors with real constant eigenvalues \cite[Thm.~1.1]{AFNP} and the spinor bundle decomposes into a direct sum of $T$-eigenbundles preserved by $\nabla^{c}$:  $\Sigma=\bigoplus_{\gamma\in{\rm Spec}(T)}\Sigma_{\gamma}$ with $\nabla^{c}\Sigma_{\gamma}\subset\Sigma_{\gamma}$, $\forall\gamma\in{\rm Spec}(T)$. Similarly, the space of sections decomposes as $\Gamma(\Sigma)=\bigoplus_{\gamma\in{\rm Spec}(T)}\Gamma(\Sigma_{\gamma})$.    Then,  from the generalized  Schr\"odinger-Lichnerowicz formula   and for the  smallest eigenvalue  $\lambda_{1}$  of the square $\slashed{D}^{2}$  restricted on $\Sigma_{\gamma}$ it follows that (see \cite{AF, ABK})
 \[\lambda_{1}\big(\slashed{D}^{2}\big|_{\Sigma_{\gamma}}\big)\geq \frac{1}{4}\Sca_{\rm min}^{g}+\frac{1}{8}\|T\|^{2}-\frac{1}{4}\gamma^{2}:=\be_{\rm univ}(\gamma).\]
  The equality occurs if and only if $\Sca^{g}$ is constant and $\varphi$ is $\nabla^{c}$-parallel, i.e. $\nabla^{g}_{X}\varphi+\frac{1}{4}(X\lrcorner T)\cdot\varphi=0$. These are the spinors fields that we are  mainly interested in here.  A very similar formula with this  estimate holds   on the whole spinor bundle $\Sigma=\oplus_{\gamma}\Sigma_{\gamma}$, where one has to consider the maximum of the possible different eigenvalues $\{\gamma_{1}^{2}, \ldots, \gamma_{q}^{2}\}$, see \cite{ABK} for details.
  
   $\nabla^{c}$-parallel or $\nabla^{c}$-harmonic  spinor fields   have been originally studied in \cite{FrIv}, see also \cite{Dalakov, IvPap} for   4-dimensional manifolds or KT-manifolds (K\"ahler with torsion), \cite{FriedG2} for   $\G_2$-manifolds  and   \cite{AF, Cas} for a description in terms of the Casimir operator. 
With the purpose of   having the results in a uniform notation, we recall that
  \begin{theorem} \label{KNOWN1}   \textnormal{(\cite[Cor.~3.2]{FrIv}, \cite[Thm.~1.1]{AFNP}, \cite[p.~306]{ABK})}
Let $\varphi_{0}$ be a   $\nabla^{c}$-parallel spinor. Then, 
\begin{eqnarray*}
\Sca^{c} \varphi_{0}&=&-2dT\cdot\varphi_{0}=-4\sigma_{T}\cdot\varphi_{0}, \quad (\ast)\\
 \Ric^{c}(X)\cdot\varphi_{0}&=&\frac{1}{2}\big(X\lrcorner dT+\nabla^{c}_{X}T\big)\cdot\varphi_{0},
\end{eqnarray*}
 in particular  the  scalar curvatures $\Sca^{c}, \Sca^{g}$ are constant.  If  in addition  $\varphi\in\Sigma_{\gamma}$ for some $\gamma\neq 0$, then 
  \[
  \Sca^{g}=2\gamma^{2}-\frac{1}{2}\|T\|^{2}, \quad \Sca^{c}=2(\gamma^{2}-\|T\|^{2}), \quad \Ric^{c}(X)=\frac{1}{2}\big(X\lrcorner dT\big)\cdot\varphi_{0}=(X\lrcorner \sigma_{T})\cdot\varphi_{0}.\]
  \end{theorem}
For convenience, we  shall often refer to $(\ast)$ as the {\it integrability condition} for the existence of  a $\nabla^{c}$-parallel spinor.  We also agree to use the notation {\it characteristic  spinor field}, for a spinor field  in the kernel of the Dirac operator induced by $\nabla^{c}$, i.e. a $\nabla^{c}$-harmonic spinor field.  Obviously, any $\nabla^{c}$-parallel spinor is characteristic, but the converse is false in general   \cite[Thm.~3.4]{FrIv}.
\subsection{$\nabla^{c}$-parallel spinors which are also Killing spinors}
 Let us  provide now a  necessary and sufficient condition   which allows us to decide when a $\nabla^{c}$-parallel spinor lying in $\Sigma_{\gamma}$  is a  real Killing spinor (with respect to the same metric $g$ that the equation $\nabla^{g}_{X}\varphi+\frac{1}{4}(X\lrcorner T)\cdot\varphi=0$ holds).  For $n\leq 8$, we see that this is the limiting case of a criterion about the existence of  this kind of spinors, see \cite[Lem.~4.1]{ABK} or the inequalities given in (\ref{inqilkas}) in Section \ref{finals}.  In particular, we show that if such a spinor exists and one of these inequalities   holds as an equality, then the spinor is a Killing spinor (see also  Proposition \ref{afterklik}).

\begin{prop}\label{klik}
  Let $(M^{n}, g, T)$   be a compact  connected   Riemannian spin manifold with $\nabla^{c}T=0$ and   positive scalar curvature,  carrying a non-trivial spinor field $\varphi_{0}\in\Gamma(\Sigma)$ such that
     \begin{equation}\label{fund1}
\nabla^{c}_{X}\varphi_{0}=\nabla^{g}_{X}\varphi_{0}+\frac{1}{4}(X\lrcorner T)\cdot\varphi_{0}=0,\quad \forall \ X\in\Gamma(TM),  \quad T\cdot\varphi_{0}=\gamma\varphi_{0}, \quad  \text{for some} \ \gamma\in\bb{R}\backslash\{0\}.
  \end{equation}
  Then, $\varphi_{0}$ is a real Killing spinor  (with respect to $g$)  if and only if  the (constant) eigenvalue  $0\neq\gamma\in\Spec(T)$ satisfies the equation
\[
 \gamma^{2}=\frac{4n}{9(n-1)}\Sca^{g}.\quad (\dagger)
 \]
If this is the case, then the  Killing number is given by  $\kappa:=3\gamma/4n$  and the following relation  hold 
     \begin{equation}\label{t1s}
(X\lrcorner T)\cdot \varphi_{0}+\frac{3\gamma}{n}X\cdot\varphi_{0}=0, \quad\forall \  X\in\Gamma(TM).
\end{equation} 
 For $n\leq 8$, the condition $(\dagger)$ is equivalent to  $\gamma^{2}=\frac{2n}{9-n}\|T\|^{2}$ or $\Sca^{g}=\frac{9(n-1)}{2(9-n)}\|T\|^{2}$ and if this is the case, then the action of the symmetric endomorphism $dT$ on $\varphi_{0}$ is given by $dT\cdot\varphi_{0}=-\frac{3\gamma^{2}(n-3)}{2n}\varphi_{0}$.
 \end{prop}
\begin{proof}
 Because $D^{c}\equiv D^{1/4}=D^{g}+\frac{3}{4}T$ and $\varphi_{0}\in\Sigma_{\gamma}$ is $\nabla^{c}$-parallel,  it follows that $\varphi_{0}$ is  an eigenspinor of the Riemannian Dirac operator $D^{g}$ (and thus also of $\slashed{D}=D^{g}+\frac{1}{4}T$):
\[
 D^{g}(\varphi_{0})=-\frac{3\gamma}{4}\varphi_{0}.
\]
 Given a compact Riemannian spin manifold with positive scalar curvature, it is well-known  that  if $\varphi_{0}\in\Gamma(\Sigma)$ is an eigenspinor of $D^{g}$ with one of the eigenvalues $\pm \frac{1}{2}\sqrt{\frac{n\Sca^{g}_{\rm min}}{n-1}}$, 
then $\varphi_{0}$ is a real Killing spinor  with corresponding Killing number $\kappa:=\mp \frac{1}{2}\sqrt{\frac{\Sca^{g}_{\rm min}}{n(n-1)}}$ (the converse is also true, see \cite[pp.~20, 31]{Baum} or \cite{Fr1980, Gr}).  However, since $\varphi_{0}$ is $\nabla^{c}$-parallel,   Theorem \ref{KNOWN1} allows us to drop the minimal condition in $\Sca^{g}$.  Hence,  let us  assume for example that $\gamma$ is positive; then $\varphi_{0}$ is   a real Killing spinor if and only if
\[
-\frac{3\gamma}{4}=-\frac{1}{2}\sqrt{\frac{n\Sca^{g}}{n-1}}, 
\]
which gives rise to the stated relation $(\dagger)$.   Then,  we also compute $\kappa=+\frac{1}{2}\sqrt{\frac{\Sca^{g}}{n(n-1)}}=3\gamma/4n$.  Similar is treated the case where $\gamma$ is negative, where one has $-\frac{3\gamma}{4}=\frac{1}{2}\sqrt{\frac{n\Sca^{g}}{n-1}}$ and the Killing  number must be given by $\kappa=-\frac{1}{2}\sqrt{\frac{\Sca^{g}}{n(n-1)}}$.   
The alternative expressions for $n<9$ easily occur using Theorem \ref{KNOWN1} and the given expression of $\gamma^{2}$; the restriction about $n$  is taken due to the positivity of $\Sca^{g}$. The relation (\ref{t1s}) is a simple consequence of the equations $\nabla^{c}_{X}\varphi_{0}=0$ and $\nabla^{g}_{X}\varphi_{0}=\frac{3\gamma}{4n}X\cdot\varphi_{0}$. Finally, for  the action of the 4-form $dT=2\sigma_{T}$ on $\varphi_{0}$, viewed as a symmetric endomorphism,  we apply  the  integrability condition for $\nabla^{c}$-parallel spinors. \end{proof}
   In particular, we observe that on a compact connected Riemannian spin manifold $(M^{n}, g, T)$, real Killing spinors with  Killing number   $\kappa=\frac{3\gamma}{4n}$,  $\nabla^{c}$-parallel spinors lying on  $\Sigma_{\gamma}$, or characteristic spinors lying in $\Sigma_{\gamma}$, if existent, are sharing a common property: {\it They are  all  eigenspinors of the Riemannian Dirac operator with the same eigenvalue $-\frac{3\gamma}{4}\neq 0$, where $0\neq \gamma\in\Spec(T)$ is a real $T$-eigenvalue. }
  \begin{example}\textnormal{Let $(M^{6}, g, J)$ be a (strict) 6-dimensional nearly K\"ahler manifold, see  Section \ref{nka} for more  details and references.  Such a manifold admits two $\nabla^{c}$-parallel spinors $\varphi^{\pm}$ lying in $\Sigma_{\gamma}$ with $\gamma=\pm 2\|T\|$ \cite{FrIv}. The scalar curvature is given by   $\Sca^{g}=\frac{15}{2}\|T\|^{2}$ and this coincides with $\frac{9(n-1)}{4n}\gamma^{2}=\frac{9(n-1)}{2(9-n)}\|T\|^{2}$.  Therefore,  by Proposition  \ref{klik} the spinors  $\varphi^{\pm}$ are real Killing spinors with $\kappa=3\gamma/4n=\pm\|T\|/4$, a well-known result by    \cite{FG, Gr}.    In particular, $dT\cdot\varphi^{\pm}=-\frac{1}{2}\Sca^{c}\cdot\varphi^{\pm}=-\frac{3\gamma^{2}(n-3)}{2n}\varphi^{\pm}$, i.e. $dT\cdot\varphi^{\pm}=-3\|T\|^{2}\varphi^{\pm}$ (see also \cite[Lem.~10.7]{FrIv}). Consider a 7-dimensional   nearly parallel $\G_2$-manifold (more details are given in Section \ref{npg2m}). There is a unique $\nabla^{c}$-parallel spinor $\varphi_{0}$ with $\gamma=-\sqrt{7}\|T\|$ and the scalar curvature   $\Sca^{g}=\frac{27}{2}\|T\|^{2}$ coincides with $\frac{9(n-1)}{4n}\gamma^{2}=\frac{9(n-1)}{2(9-n)}\|T\|^{2}$. Thus,   $\varphi_{0}$ must be  a Killing spinor with $\kappa=3\gamma/4n=-\frac{3}{4\sqrt{7}}\|T\|$, which is well-known by \cite{FKMS, FrIv}.  In particular, $dT\cdot\varphi_{0}=-\frac{3\gamma^{2}(n-3)}{2n}\varphi_{0}$, i.e. $dT\cdot\varphi_{0}=-6\|T\|^{2}\varphi_{0}$ (see also \cite[Ex.~5.2]{FrIv}).}
  \end{example}
  \smallskip
 We proceed with some further properties of $\nabla^{c}$-parallel spinors lying in $\Sigma_{\gamma}$ for some $\gamma\neq 0$.
        \begin{lemma}\label{return}
 Assume that $(M^{n}, g, T)$ is as in Proposition \ref{klik} and that $n=\dim M>3$. Let     $0\neq\varphi_{0}\in\Gamma(\Sigma)$  be a non-trivial spinor  satisfying the equations given in (\ref{fund1}). 
  Then,  $(X\wedge T)\cdot\varphi_{0}\neq 0$ needs to hold for at least a vector field $X$. 
  \end{lemma}
  \begin{proof}
    If $\nabla^{c}\varphi_{0}=0$, then using the identity $X\wedge T=X\lrcorner T+X\cdot T$ (see \cite[p.~14]{Baum} or \cite[Appen.~C]{ABK}) we get 
     \[
 \nabla^{g}_{X}\varphi_{0}= \frac{\gamma}{4}X\cdot\varphi_{0}-\frac{1}{4}(X\wedge T)\cdot\varphi_{0}, \quad \forall X\in\Gamma(TM).
 \]
If  $(X\wedge T)\cdot\varphi_{0}=0$ for any $X$,   then  $\nabla^{g}_{X}\varphi_{0}= \frac{\gamma}{4}X\cdot\varphi_{0}$ for any $X$, which means that $\varphi_{0}$ is real Killing spinor    with  Killing number $\kappa=\gamma/4$. Thus the relation $\gamma^{2}=16\kappa^{2}$ needs to be true. We  prove   that this is contradiction.  Indeed, based on the above discussion  and using the values $\mp \frac{1}{2}\sqrt{\frac{\Sca^{g}}{n(n-1)}}$ we get the relation $\gamma^{2}=16\kappa^{2}=16\cdot \frac{\Sca^{g}}{4n(n-1)}=\frac{4\Sca^{g}}{n(n-1)}.$  But then, by Theorem \ref{KNOWN1} it follows that $\big(1-\frac{8}{n(n-1)}\big)\Sca^{g}=-\frac{1}{2}\|T\|^{2}$, which is a contradiction for any $n>3$, since $\Sca^{g}>0$ and $\|T\|^{2}>0$ ($T\neq 0$ by assumption). 
Finally,   another direct  argument   applies. If $\varphi_{0}$ is a Killing spinor with  $\kappa=\gamma/4$, then it must be an eigenspinor of the Riemannian Dirac operator $D^{g}$ with eigenvalue $-n\gamma/4$.  However, $\varphi_{0}$ is $\nabla^{c}$-parallel and belongs to $\Sigma_{\gamma}$, hence   it is also $D^{g}(\varphi_{0})=-\frac{3\gamma}{4}\varphi_{0}$. Thus, for $n>3$ this leads to a  contradiction.
\end{proof}
  

 Of course,   the Clifford multiplication of the 4-form $(X\wedge T)$ with a $\nabla^{c}$-parallel spinor $\varphi$ does not need to be an element of the subbundle  $\{X\cdot\varphi : X\in TM\}\subset \Sigma$.  In dimensions $6$ and $7$, for nearly K\"ahler manifolds and   nearly parallel $\G_2$-manifolds,  respectively,  one can prove that there exists an appropriate constant $c$ that makes the equation  $(X\wedge T)\cdot\varphi=c X\cdot\varphi$ true for any vector field $X$ and this corresponds to  the deeper  one-to-one correspondence between $\nabla^{c}$-parallel spinors and  Riemannian Killing spinors \cite{FrIv}.

\smallskip
       \begin{lemma}\label{nice}
 Let $0\neq\gamma\in\Spec(T)$ be a real non-zero $T$-eigenvalue.  Then, a real Killing spinor $\varphi\in\cal{K}(M^{n}, g)_{\frac{3\gamma}{4n}}$  is  characteristic  i.e. $D^{c}(\varphi)=0$, if and only if $\varphi\in\Sigma_{\gamma}$.  
   \end{lemma}
 \begin{proof}
 Consider the Dirac operator $D^{c}\equiv D^{1/4}$ associated to the characteristic connection. Then
$ D^{c}(\varphi)=D^{g}(\varphi)+\frac{3}{4}T\cdot\varphi=-\frac{3\gamma}{4}\varphi+\frac{3}{4}T\cdot\varphi$, and the same occurs using the fact that $\nabla^{c}_{X}\varphi=\frac{1}{4}(X\lrcorner T)\cdot\varphi+\frac{3\gamma}{4n}X\cdot\varphi$:
\begin{eqnarray*}
 D^{c}(\varphi)&=&\sum_{i} e_{i}\cdot \nabla^{c}_{e_{i}}\varphi=\sum_{i}e_{i}\cdot\{\frac{1}{4}(e_{i}\lrcorner T)\cdot\varphi+\frac{3\gamma}{4n}e_{i}\cdot\varphi\}=\frac{1}{4}\sum_{i}e_{i}\cdot (e_{i}\lrcorner T)\cdot\varphi+\frac{3\gamma}{4n}\sum_{i}e_{i}\cdot e_{i}\cdot\varphi\\
 &=&
 \frac{3}{4}T\cdot\varphi-\frac{3\gamma}{4}\varphi=\frac{3}{4}(T\cdot\varphi-\gamma\varphi).
\end{eqnarray*}
   \end{proof}
Athough on a compact Riemannian manifold with constant scalar curvature given by $\Sca^{g}=\frac{9(n-1)}{4n}\gamma^{2}$ for some non-zero real eigenvalue $\gamma\in\Spec(T)$, any existent $\nabla^{c}$-parallel spinor   $\varphi\in\Sigma_{\gamma}$ must be also a Killing spinor field, conversely we cannot claim that a  Killing spinor with $\kappa=3\gamma/4n$  is $\nabla^{c}$-parallel, even if it is characteristic.   From now on, we agree to denote by $\ke\big(\nabla^{c}\big)$ and $\ke\big(D^{c}\big)$ the following sets:
\[
 \ke\big(\nabla^{c}\big):=\{\varphi\in\Gamma(\Sigma_{\gamma})\subset\Gamma(\Sigma) : \nabla^{c}\varphi\equiv 0\}, \quad \ke\big(D^{c}\big):=\{\varphi\in\Gamma(\Sigma_{\gamma})\subset\Gamma(\Sigma) : D^{c}(\varphi)=0\}.
 \] 
 With the aim to construct the desired one-to-one correspondence between spinor fields in $\ke\big(\nabla^{c}\big)$ and 
$\ke\big(D^{c}\big)$, it is necessary the eigenvalues of the endomorphism   $dT+\frac{1}{2}\Sca^{c}$ on $\Sigma$ to be non-negative; this is a classical result of   Th.~Friedrich and  S.~Ivanov  \cite[Thm.~3.4]{FrIv}. In particular, for a Killing spinor $\varphi\in\cal{K}(M, g)_{\frac{3\gamma}{4n}}$ with $\varphi\in\Sigma_{\gamma}$  the relation $D^{c}(T\cdot\varphi)+T\cdot D^{c}(\varphi)=-2\sum_{i}(e_{i}\lrcorner T)\cdot \nabla^{c}_{e_{i}}\varphi$ 
 implies that  $\sum_{i}(e_{i}\lrcorner T)\cdot \nabla^{c}_{e_{i}}\varphi=0$ (see \cite[Thm.~3.3]{FrIv}). Thus, integrating   the  generalized Schr\"odinger-Lichnerowicz formula associated to $D^{c}$,  it follows that 
\[
  0=\int_{M}\Big[\|\nabla^{c}\varphi\|^{2}+\frac{1}{2}\langle dT\cdot\varphi+\frac{\Sca^{c}}{2}\varphi, \varphi\rangle\Big].
  \]
Combining now with Lemma \ref{nice} and  Proposition \ref{klik}, one deduces 

\begin{theorem}\label{newFI}
Let $(M^{n}, g, T)$ be  compact connected Riemannian spin manifold $(M^{n}, g, T)$, with $\nabla^{c}T=0$ and   positive scalar curvature given by $\Sca^{g}=\frac{9(n-1)\gamma^{2}}{4n}$ for some constant $0\neq\gamma\in\Spec(T)$. If the symmetric endomorphism $dT+\frac{1}{2}\Big[\frac{9(n-1)}{4n}\gamma^{2}-\frac{3}{2}\|T\|^{2}\Big]$ acts on $\Sigma$ with non-negative eigenvalues, then the following classes of spinors, if existent, coincide  
   \[
  \ke\big(\nabla^{c}\big) \cong \ke\big(D^{c}\big)\cong \bigoplus_{\gamma\in\Spec(T)}\Big[\Gamma(\Sigma_{\gamma})\cap \cal{K}(M^{n}, g)_{\frac{3\gamma}{4n}}\Big].
  \]
\end{theorem}

\subsection{A basic one-to-one correspondence} 
Next we  extend the correspondence established in Theorem \ref{newFI} to $\nabla^{c}$-parallel twistor and Killing spinors with torsion.  A special version of the result that we discuss below, namely for  nearly K\"ahler manifolds in dimension 6 and only for a specific parameter $s=5/12$, is known by \cite[Thm.~6.1]{ABK}.  Here,  following  a different method  we  present a   more general version.  

Notice that any twistor spinor with torsion with respect to  $\nabla^{c}=\nabla^{1/4}$ which is   characteristic $D^{c}(\varphi)=0$, is in fact   $\nabla^{c}$-parallel. This also occurs    after applying Lemma \ref{good} for a characteristic twistor spinor;  $\frac{1}{2}\Sca^{s}\varphi=-4s(3-4s)\sigma_{T}\cdot\varphi$ and for $s=1/4$ one has $\frac{1}{2}\Sca^{c}\cdot\varphi=-2\sigma_{T}\cdot\varphi=-dT\cdot\varphi$ (notice that in dimensions $n=3, 4$ the  4-form $\sigma_{T}$ is identically equal to zero). Thus, from now on we are mainly interested in twistors with torsion for some $s\neq 1/4$.

       \begin{theorem}\label{general1}
 Let $(M^{n}, g,T)$ be a compact connected Riemannian spin manifold  with $\nabla^{c}T=0$ and  assume that $\varphi\in\Gamma(\Sigma)$ is  a spinor field such that $\nabla^{c}\varphi=0$, where   $\nabla^{c}=\nabla^{g}+\frac{1}{2}T$ is the characteristic connection. Let $\gamma\in\bb{R}\backslash\{0\}$ be a non-zero real number. Then, the following  conditions are equivalent:
 
$(a)$ $\varphi\in\Gamma(\Sigma_{\gamma})\cap\ke(P^{s}):=\ke(P^{s}\big|_{\Sigma_{\gamma}})$  with respect to the family $\{\nabla^{s} : s\in\bb{R}\backslash\{1/4\}\}$, 

$(b)$ $\varphi\in \cal{K}^{s}(M, g)_{\zeta}$  with respect to  the family $\{\nabla^{s} : s\in\bb{R}\backslash\{0, 1/4\}\}$ with $\zeta=3(1-4s)\gamma/4n$,

$(c)$ $ \varphi\in \cal{K}(M, g)_{\kappa}$ with $\kappa={3\gamma}/{4n}$.
 \end{theorem}
 
\begin{proof}
By definition, it is $\nabla^{c}_{X}\varphi=\nabla^{g}_{X}\varphi+\frac{1}{4}(X\lrcorner T)\cdot\varphi$, hence we write
\begin{equation}\label{nsc}
\nabla^{s}_{X}\varphi=\nabla^{c}_{X}\varphi+\frac{4s-1}{4}(X\lrcorner T)\cdot\varphi,  \quad \forall \  X\in\Gamma(TM), \ \varphi\in\Gamma(\Sigma).
\end{equation}
Because $D^{c}\equiv D^{1/4}=D^{g}+\frac{3}{4}T$,   we also conclude that
\[
D^{s}(\varphi)=D^{g}(\varphi)+3sT\cdot\varphi=D^{c}(\varphi)+\frac{3(4s-1)}{4}T\cdot\varphi.
\]
Then, the twistor spinor equation  with respect to $\nabla^{s}$  can be expressed by  
 \[\nabla^{c}_{X}\varphi+\frac{4s-1}{4}(X\lrcorner T)\cdot\varphi+\frac{1}{n}X\cdot\big\{D^{c}(\varphi)+\frac{3(4s-1)}{4}T\cdot\varphi\big\}=0, \quad \forall \ X\in\Gamma(TM).
 \]
Assume now that $\varphi\in\Gamma(\Sigma)$ is a twistor spinor with respect to $\nabla^{s}$  which verifies the  equations   (\ref{fund1}).  %
  Then, for $s\neq \frac{1}{4}$,  we conclude that the twistor equation  is equivalent to the relation (\ref{t1s}),   namely  
 \[
 (X\lrcorner T)\cdot\varphi=-\frac{3\gamma}{n}X\cdot\varphi,  
 \]
Thus,  using (\ref{nsc})  we can easily see that $\varphi$ is a Killing  spinor with torsion with respect to the family $\{\nabla^{s} : s\neq 0, \frac{1}{4}\}$ with Killing number $\zeta:=3(1-4s)\gamma/4n$ and moreover a Riemannian Killing spinor with Killing number $\kappa=3\gamma/4n$:  
\[
\nabla^{s}_{X}\varphi=-\frac{3\gamma(4s-1)}{4n}X\cdot\varphi, \quad\text{and}\quad \nabla^{g}_{X}\varphi=\frac{3\gamma}{4n}X\cdot\varphi.
\]
 This proves the one direction $(a)\Rightarrow(b)\Rightarrow (c)$.  Consider now some $\varphi\in\cal{K}^{s}(M, g)_{\zeta}$ with $s\neq 0, \frac{1}{4}$ and $\zeta:=\frac{3\gamma(1-4s)}{4n}$   for some real parameter $\gamma\neq 0$.  Under the additional  assumption $\nabla^{c}\varphi=0$, we  will show  that $\varphi$     is an eigenspinor of  $T$, i.e.  $T\cdot\varphi=\gamma\varphi$.  Indeed, any Killing spinor field with  torsion is also a twistor spinor with torsion, hence   the twistor equation  yields the relation 
\[ \frac{4s-1}{4}(X\lrcorner T)\cdot\varphi+\frac{3(4s-1)}{4n}X\cdot T\cdot \varphi=0. \quad (\sharp)\]
On the other hand, for any $s\neq 0, \frac{1}{4}$ it holds that $\nabla^{s}_{X}\varphi=\zeta X\cdot\varphi$ with $\zeta:=\frac{3\gamma(1-4s)}{4n}\neq0$. Due to  (\ref{nsc}) we finally get
\[
\frac{4s-1}{4}(X\lrcorner T)\cdot\varphi=\zeta X\cdot\varphi, 
\]
for any  vector field $X$. Inserting this in $(\sharp)$ we see that
\[
X\cdot (\zeta \varphi+\frac{3(4s-1)}{4n} T\cdot \varphi)=0,  
\]
and our claim   follows. This shows the direction $(b)\Rightarrow (a)$ and remains to  prove also that   $(c)$ implies $(b)$.  Assume that $\varphi\in\cal{K}(M, g)_{\kappa}$ is a real Killing spinor  with Killing number $\kappa:=3\gamma/4n$ and such that $\nabla^{c}_{X}\varphi=0$ for any vector field $X$. Then, $\nabla^{g}_{X}\varphi=-\frac{1}{4}(X\lrcorner T)\cdot\varphi=\frac{3\gamma}{4n}X\cdot\varphi$, and  using (\ref{nsc}) we complete the proof. 
\end{proof}
\noindent{\bf Remarks:}  Choosing one of the conditions $(b)$ or $(c)$ in Theorem \ref{general1} for some real constant $\gamma\neq 0$, we see that the relation $T\cdot\varphi=\gamma\varphi$ follows by the twistor equation.   We also  remark that one can start with a  triple  $(M^{n}, g, T)$ with $\nabla^{c}T=0$, endowed with a real Killing spinor $\varphi\in\cal{K}(M^{n}, g)_{\frac{3\gamma}{4n}}\cap\Gamma(\Sigma_{\gamma})$ for some $0\neq \gamma\in\Spec(T)$,  and  similarly prove  that the conditions $\varphi\in\ke(\nabla^{c})$, $\varphi\in \cal{K}^{s}(M, g)_{\zeta}$ with $\zeta=3\gamma(1-4s)/4n$  and  $\varphi\in\ke(P^{s}\big|_{\Sigma_{\gamma}})$ are equivalent  one to each other. 
 
 

 \smallskip
A first  immediate  corollary of Theorem \ref{general1} is the following one:
\begin{corol}\label{appok}
 Let $(M^{n}, g, T)$ a triple with $\nabla^{c}T=0$, carrying a $\nabla^{c}$-parallel spinor field $0\neq\varphi$ satisfying one of the conditions (a), (b) or (c) in Theorem \ref{general1}.  Then, the following holds for any vector field $X$
\[ (X\wedge T)\cdot\varphi=\frac{(n-3)\gamma}{n}X\cdot\varphi.\]
 Conversely,  if $\varphi\in\Sigma_{\gamma}$ is a $\nabla^{c}$-parallel spinor satisfying the previous relation for any $X\in\Gamma(TM)$ and for some real $0\neq\gamma\in\Spec(T)$,  then  the conditions  given in Theorem \ref{general1} must hold, in particular $\varphi$ is a real Kllling spinor with respect to $g$  with Killing number $\kappa=\frac{3\gamma}{4n}$.   Finally, for $n=3$  it is  $(X\wedge T)\cdot\varphi=0$ identically, i.e. $(X\lrcorner T)\cdot\varphi=-\gamma X\cdot\varphi=-X\cdot T\cdot\varphi$.  \end{corol}
 \begin{proof}
 The key ingredient  of the  proof is  encoded in  (\ref{t1s}).  Given a $\nabla^{c}$-parallel spinor  satisfying one of the conditions in Theorem \ref{general1}, then $\varphi\in\Sigma_{\gamma}$ for some $\gamma\neq 0$ and the relation (\ref{t1s}) needs to hold. Thus, the result  easily follows due to the identity $X\wedge T=X\lrcorner T+X\cdot T$. For $n=3$ we get $(X\wedge T)=0$ and so  $X\lrcorner T=-X\cdot T$ (as it should be for dimensional reasons). 
 \end{proof}
Now, similar with Lemma \ref{nice} we observe  that
 \begin{lemma}\label{nice2}
 Let $0\neq \gamma\in\Spec(T)$ be a non-zero real $T$-eigenvalue. Then the following hold:
 
 (a) A Killing spinor with torsion $\varphi\in\cal{K}^{s}(M, g)_{\frac{3\gamma(1-4s)}{4n}}$ for some $s\neq 0, 1/4$ is characteristic, if and only $\varphi\in\Sigma_{\gamma}$.
 
 (b) A twistor spinor with torsion $\varphi\in\ke\big(P^{s}|_{\Sigma_{\gamma}}\big)$  for some $s\neq  0, 1/4$ is characteristic, if and only if $\varphi$ is a $D^{s}$-eigenspinor with eigenvalue $-\frac{3\gamma(1-4s)}{4}$ i.e. $\varphi\in\cal{K}^{s}(M, g)_{\frac{3\gamma(1-4s)}{4n}}\cap \Gamma(\Sigma_{\gamma})$. In particular, for $s=0$, a twistor spinor $\varphi\in\ke\big(P^{g}|_{\Sigma_{\gamma}}\big)$ is characteristic, if and only if $D^{g}(\varphi)=-\frac{3\gamma}{4}\varphi$, i.e.  $\varphi\in\cal{K}(M, g)_{\frac{3\gamma}{4n}}\cap\Gamma(\Sigma_{\gamma})$. 
 \end{lemma}
 \begin{proof}
 (a) The first claim is an simple consequence of $D^{s}(\varphi)=D^{c}(\varphi)+\frac{3(4s-1)}{4}T\cdot\varphi$ and the fact that $\varphi$ is an eigenspinor of $D^{s}$ with eigenvalue $-\frac{3\gamma(1-4s)}{4}$.
 
 (b) Consider a twistor spinor with torsion $\varphi\in\ke(P^{s})$  for some $s\neq 0, 1/4$ such that $T\cdot\varphi=\gamma\varphi$. We write the twistor equation as
\[
 \nabla^{s}_{X}\varphi+\frac{1}{n}X\cdot D^{c}(\varphi)+\frac{3(4s-1)}{4n}X\cdot T\cdot\varphi=0,
\]
 which is equivalent to
 \[
 \frac{1}{n}X\cdot\big(D^{c}-D^{s}\big)(\varphi)=-\frac{3\gamma(4s-1)}{4n}X\cdot\varphi.
 \]
  Since the latter equation holds for any vector field $X$, we easily conclude.  A more direct way is given as follows: Suppose that   $\varphi\in\ke(P^{s}\big|_{\Sigma_{\gamma}})$ is in addition characteristic, i.e. $D^{c}(\varphi)=0$. Then
   \[
 \nabla^{s}_{X}\varphi=-\frac{1}{n}X\cdot D^{s}(\varphi)=-\frac{1}{n}X\cdot\big[D^{c}(\varphi)+\frac{3(4s-1)}{4}T\cdot\varphi\big]=\frac{3\gamma(1-4s)}{4n}X\cdot\varphi,
 \]
  i.e.  $\varphi\in\cal{K}^{s}(M, g)_{\frac{3\gamma(1-4s)}{4n}}$ and the converse is obvious. Similarly for $s=0$.
 \end{proof}
Combining Theorem \ref{newFI} with Theorem \ref{general1} and Lemma \ref{nice2}, we take the following extension.
  \begin{corol}\label{addd} 
 Let $(M^{n}, g, T)$ be  compact connected Riemannian spin manifold $(M^{n}, g, T)$, with $\nabla^{c}T=0$ and  positive scalar curvature given by $\Sca^{g}=\frac{9(n-1)\gamma^{2}}{4n}$ for some constant $0\neq\gamma\in\Spec(T)$. If the symmetric endomorphism  $dT+\frac{1}{2}\Big[\frac{9(n-1)}{4n}\gamma^{2}-\frac{3}{2}\|T\|^{2}\Big]$ acts on $\Sigma$ with non-negative eigenvalues, then the following classes of spinors, if existent, coincide 
   {\small{\[
 \ke(\nabla^{c})\cong \bigoplus_{\gamma\in\Spec{(T)}}\Big[\Gamma(\Sigma_{\gamma})\cap\cal{K}(M, g)_{\frac{3\gamma}{4n}}\Big]\cong\bigoplus_{\gamma\in\Spec{(T)}} \Big[\Gamma(\Sigma_{\gamma})\cap\cal{K}^{s}(M, g)_{\frac{3(1-4s)\gamma}{4n}}\Big]\cong \bigoplus_{\gamma\in\Spec{(T)}} \Big[\ke(P^{s}\big|_{\Sigma_{\gamma}})\cap\ke\big(D^{c}\big)\Big].
 \]}}
 Here, the parameter  $s$ takes values in $\bb{R}\backslash\{0, 1/4\}$ for the third set, and for the last set it is $s\in\bb{R}\backslash\{1/4\}$. 

 \end{corol}

\section{Examples}\label{exex} 
The most representative classes of special structures for which  Theorems \ref{newFI}, \ref{general1} and  Corollaries \ref{appok}, \ref{addd} make sense, are    6-dimensional    nearly K\"ahler manifolds  and 7-dimensional   nearly parallel $\G_2$-manifolds. Of course the same holds for Proposition \ref{klik}, which was the starting  point of this theoretical approach.   Let us describe these special structures in some detail.
\subsection{Nearly K\"ahler manifolds and their spinorial properties}\label{nka}
  
         
    A nearly K\"ahler manifold   is an almost Hermitan manifold endowed with an almost complex structure $J$ such that  $(\nabla^{g}_{X}J)X=0$.  Next we need to recall basic results from  \cite{FrIv, FG, Gr, Baum, Cas}, where we refer    for more details and proofs.  
  In dimension 6 strict nearly K\"ahler   manifolds $(M, g, J)$  are very special;   they are spin,  the  first Chern class vanishes $c_{1}(M^{6}, J)=0$ and $g$ is an Einstein metric. In the homogeneous case, 6-dimensional nearly K\"ahler manifolds  are exhausted by the 3-symmetric spaces $\bb{S}^{6}=\G_2/\SU_{3}$, $\bb{C}P^{3}=\SO_{5}/\U_2=\Sp_{2}/(\Sp_{1}\times \U_{1}),  \ \bb{F}_{1, 2}=\SU_{3}/T_{\rm max}$ and   $\Ss^{3}\times\Ss^{3}=\SU_{2}\times\SU_{2}$, endowed with   a naturally reductive  (Einstein) metric  \cite{Butr}. Together with the  standard spheres $\Ss^{2m}$,    these spaces     exhaust all even-dimensional Riemannian manifolds  admitting real Killing spinors. 
Notice that recently  in \cite{Cortes}, a {\it locally homogeneous} nearly K\"ahler manifold of the form $M=\tilde{M}/\Gamma$ was described, where $\tilde{M}=\Ss^{3}\times\Ss^{3}$ and $\Gamma$  is any finite  subgroup of $\SU_{2}\times\SU_{2}$.
 
   Now, any nearly K\"ahler manifold  admits a characteristic connection $\nabla^{c}$ with parallel skew-torsion, given by $T(X, Y):=(\nabla^{g}_{X}J)JY$ (Gray connection) \cite[Thm.~10.1]{FrIv}.  In particular,  in dimension 6  there exists a positive  constant  $\tau_{0}\neq 0$ such that  (see \cite{FrIv, Cas, AFNP})
     \[
     \|T\|^{2}=2\tau_{0},\quad \Sca^{g}=15\tau_{0},\quad \Sca^{c}=12\tau_{0},\quad  \Ric^{g}=\frac{5}{2}\tau_{0}\Id.
     \] 
 Notice that working with an even dimensional manifold $M^{2m}$ of constant positive  scalar curvature,    the spinor bundle splits $\Sigma=\Sigma^{+}\oplus\Sigma^{-}$  and  there is a bijection between the subbundles 
  \[
  E_{\pm}=\{\varphi\in\Gamma(\Sigma) : \nabla^{g}_{X}\varphi\pm \kappa X\cdot\varphi=0, \  \forall X\in \Gamma(TM)\},
  \]
    given by the map $\varphi^{+}:=\psi^{+}+\psi^{-}\mapsto\varphi^{-}:=\psi^{+}-\psi^{-}$ for some $\varphi^{+}\in E_{+}$ and $\psi^{\pm}\in\Sigma^{\pm}$, where $\kappa$ is given as in the proof of Proposition \ref{klik}. 
 If $M^{2m}\neq\Ss^{2m}$, then $\dim_{\bb{C}}E_{\pm}\leq 2^{m-1}=\dim_{\bb{C}}\Delta_{2m}^{\pm}$ while for the standard spheres $\Ss^{2m}$ we have the characterization  $\dim_{\bb{C}}E_{\pm}=2^{m}$ \cite{Fr1980}.   
  For a 6-dimensional nearly K\"ahler manifold it is well-known   that there are two Riemannian Killing spinors $\varphi^{\pm}$, i.e. $\dim_{\bb{C}}E_{+}=\dim_{\bb{C}}E_{-}=1$ \cite{FG, Gr}. Moreover, $\varphi^{\pm}$  are $T$-eigenspinors with eigenvalues $\gamma:=\pm 2\|T\|$ and exhaust all  $\nabla^{c}$-parallel spinors  \cite[p.~333]{FrIv}.
 
  \begin{theorem}\label{nk}
On a $6$-dimensional nearly K\"ahler manifold $(M^{6}, g, J)$ endowed with its characteristic connection $\nabla^{c}$, the following classes of spinor fields coincide:

(1) TsT with respect to the family $\{\nabla^{s} : s\in\bb{R}\backslash\{1/4\}\}$, lying in  $\Sigma_{\pm 2\|T\|}$,

(2) KsT with respect to the family $\{\nabla^{s} : s\in\bb{R}\backslash\{0, 1/4\}\}$,  with    $\zeta:= \mp\frac{(4s-1)}{4}\|T\|$, 
 
(3) Riemannian  Killing spinors,

(4) $\nabla^{c}$-parallel spinors.
\end{theorem}
\begin{proof}
Although one can immediately apply Theorem \ref{general1}, let us follow a bit different approach.  Assume that $\varphi^{\pm}$ are TsT with respect the family $\nabla^{s}$, for some $s\neq  1/4$. For $n=6$ and for $s=5/12$  it is known  \cite[Cor.~6.1]{ABK} that an element $\varphi\in\ke(P^{5/12}\big|_{\Sigma_{\gamma}})$  satisfies the equation 
 \[
\nabla^{c}_{X}\varphi-\frac{1}{18}\Big(\gamma+2\frac{\|T\|^{2}}{\gamma}\Big)X\cdot\varphi+\frac{1}{6}(X\wedge T)\cdot\varphi=0,
\]
  and this is also the Killing equation with torsion.
Now,  $\varphi^{\pm}\in\Sigma_{\pm 2\|T\|}$ are  TsT for some $s\neq 1/4$. Thus, one may assume without loss of generality  that $s=5/12$ and then $\varphi^{\pm}$  ought to satisfy  the previous equation  as well.  Because $\nabla^{c}\varphi^{\pm}=0$,  this finally reduces to $ \pm \|T\| X\cdot\varphi^{\pm}=(X\wedge T)\cdot\varphi^{\pm}$
or equivalently
  \begin{equation}\label{wrong}
  (X\lrcorner T)\cdot\varphi^{\pm}=\pm \|T\| X\cdot\varphi^{\pm} \mp 2\|T\| X\cdot\varphi^{\pm}=\mp \|T\| X\cdot\varphi^{\pm}.
     \end{equation}
Of course, this is exactly what one gets after running our twistor equation, i.e. apply (\ref{t1s}) for $\gamma=\pm2\|T\|$.
   Hence, a simple application of (\ref{nsc})  yields the result: $\nabla^{s}_{X}\varphi^{\pm}=\frac{(4s-1)}{4}(X\lrcorner T)\cdot\varphi^{\pm}=\mp\frac{(4s-1)}{4}\|T\|X\cdot\varphi^{\pm}$,    and similar for the Levi-Civita connection. 
    \end{proof}
\noindent{\bf Remarks:} In \cite[Thm.~6.1]{ABK}  the Killing number with torsion is given by $\mp\frac{\|T\|}{6}$ and this coincides with the statement  of Theorem \ref{nk}  for $s=5/12$.  In this way,  we generalise this result  by extending  the correspondence   to any real number $s\neq 0, 1/4$.  
 Moreover, and relative to Corollary \ref{appok}, notice that   any vector field $X$ satisfies $(X\wedge T)\cdot\varphi^{\pm}\neq 0$, in particular $\pm \|T\| X\cdot\varphi^{\pm}=(X\wedge T)\cdot\varphi^{\pm}$.

  \subsection{Nearly parallel $\G_2$-manifolds and their spinorial properties}\label{npg2m}

    A 7-dimensional oriented Riemannian manifold  $(M^{7}, g)$ is called a $\G_2$-manifold whenever the structure group of its frame bundle is contained in $\G_2\subset\SO_{7}$. The existence of such a reduction amounts to the existence of a  generic 3-form $\omega$.  Since $\G_{2}$ preserves $\omega$ and the same time a unit spinor $\varphi_{0}\in\Delta_{7}$,   they ought to  induce the same data, namely  (we refer to \cite{Br1, Srni, FrIv, FriedG2} for details on $\G_2$-structures)
    \[
    \omega(X, Y, Z):=(X\cdot Y\cdot Z\cdot\varphi_{0}, \varphi_{0}).
    \] 
The $\Gl_{7}$-orbit of $\omega$ in $\bigwedge^{3}(\bb{R}^{7})$, is an open set  which we shall denote by $\bigwedge^{3}_{+}(\bb{R}^{7})$.  Sections of the bundle $\bigwedge^{3}_{+}TM:=\cup_{x\in M}\bigwedge^{3}_{+}(T^{*}_{x}M)$ are call stable 3-forms and it is well-known that there is a bijection between $\G_2$-structures on $M$ and  sections  $\omega\in\Gamma(\bigwedge^{3}_{+}TM):=\Omega^{3}_{+}(M)$ \cite{FKMS}. Given such a 3-form it determines a Riemannian metric and induces an orientation on $M$.     A nearly parallel $\G_2$-structure on  $M^{7}$ is a $\G_2$-structure $\omega\in\Omega^{3}_{+}(M)$ satisfying the differential equation $d\omega=-\tau_{0}\ast \omega$ for some real constant $\tau_{0}\neq 0$.  The existence of   such a structure  is equivalent with the existence of a spin structure carrying  a real Killing spinor   \cite{FKMS, FrK}. 




 
 A nearly parallel $\G_2$-manifold admits a unique characteristic connection $\nabla^{c}$ with parallel skew-torsion $T$ given by $T:=\frac{1}{6}(d\omega, \ast\omega)\cdot\omega$ \cite[Cor.~4.9]{FrIv}. The positive real number $\tau_{0}$ links $T$ and $\omega$, in particular it holds that (see \cite{FrIv, Cas, AFNP}):
 \[
 T=-\frac{\tau_{0}}{6}\omega, \quad \|T\|^{2}=\frac{7}{36}\tau_{0}^{2}, \quad  \Ric^{g}=\frac{3}{8}\tau_{0}^{2}\Id, \quad  \Sca^{g}=\frac{21}{8}\tau_{0}^{2},   \quad \Sca^{c}=\frac{7}{3}\tau_{0}^{2}.
\]
 The connection $\nabla^{c}$ admits a unique parallel  spinor field $\varphi_{0}$ of length one such that  (here we work with 3-form $\omega$ such that $\omega\cdot\varphi_{0}=7\varphi_{0}$, see  \cite[Lem.~2.3]{ACFH})
 \begin{equation}\label{teigen}
 T\cdot \varphi_{0}=-\frac{7\tau_{0}}{6}\varphi_{0}=-\sqrt{7}\|T\|\varphi_{0}.
 \end{equation}
  In fact,  $\varphi_{0}$ is a real Killing spinor \cite{FKMS, FrIv}, with Killing number $\kappa=-\frac{1}{2}\sqrt{\frac{\Sca^{g}}{n(n-1)}}= -\frac{\tau_{0}}{8}=-\frac{3}{4\sqrt{7}}\|T\|$ and an eigenspinor of $D^{g}$ with eigenvalue $\frac{1}{2}\sqrt{\frac{n\Sca^{g}}{n-1}}=\frac{7\tau_{0}}{8}=\frac{3\sqrt{7}}{4}\|T\|$. This can be   seen also as follows:
  \[
  D^{g}\varphi_{0}=D^{c}\varphi_{0}-\frac{3}{4}T\cdot\varphi_{0}=-\frac{3}{4}T\cdot\varphi_{0}=\frac{3\sqrt{7}}{4}\|T\|\varphi_{0}.
  \] 
 Therefore, after applying  Theorem \ref{general1} one deduces that
    \begin{theorem}\label{nG2}
 On a    nearly-parallel $\G_2$-manifold $(M^{7}, g, \omega)$ endowed with its characteristic connection $\nabla^{c}$,   the following classes of spinor fields coincide:
 
 (1) TsT with respect to the family $\{\nabla^{s} : s\in\bb{R}\backslash\{1/4\}\}$, lying  in $\Sigma_{-\frac{7\tau_{0}}{6}}\equiv\Sigma_{-\sqrt{7}\|T\|}$,

 (2)  KsT with respect to the family $\{\nabla^{s} : s\in\bb{R}\backslash\{0, 1/4\}\}$, with   $\zeta:=\frac{(4s-1)\tau_{0}}{8}=\frac{3(4s-1)\|T\|}{4\sqrt{7}}$,
 
 (3)  Riemannian Killing spinors,

 (4) $\nabla^{c}$-parallel spinors.
  \end{theorem}
 \begin{proof}
 We shall present a proof using slightly different arguments.  Given some global non-trivial spinor $\varphi_{0}$ of the (real)  spin representation $\Delta_{7}\cong\bb{R}^{8}$, one has  the decomposition $\Delta_{7}=\bb{R}\varphi_{0}\oplus\{X\cdot\varphi_{0} : X\in\bb{R}^{7}\}$. On a  nearly parallel $\G_2$-manifold, the spinor $\varphi_{0}$   is the unique $\nabla^{c}$-parallel spinor \cite{FrIv, Cas}. The induced 3-form  $\omega(X, Y, Z)=(X\cdot\ Y\cdot Z\cdot\varphi_{0}, \varphi_{0})$ is such that $\|\omega\|=\sqrt{7}$, $\omega\cdot\varphi_{0}=7\varphi_{0}$, see for example \cite{ACFH}.   Now,  the space $\fr{so}(7)\cong \Lambda^{2}(\bb{R}^{7})$ decomposes under the $\G_2$-action as $\Lambda_{7}^{2}\oplus\fr{g}_{2}$, where $\Lambda_{7}^{2}=\{X\lrcorner \omega : X\in\bb{R}^{7}\}$. Then, one deduces  that $(X\lrcorner \omega)\cdot\varphi_{0}$ must be proportional to $X\cdot\varphi_{0}$,  in particular  (see \cite[Lem.~2.3]{ACFH} or \cite[p.~33]{AHol})
  \[
  \quad (X\lrcorner \omega)\cdot\varphi_{0}=-3X\cdot\varphi_{0}, \quad\forall \ X\in\Gamma(TM).
 \]
 Then, a simple  combination with $T=-\frac{\tau_{0}}{6}\omega$ gives rise to
 \begin{equation}\label{xtg2}
 (X\lrcorner T)\cdot\varphi_{0}=\frac{\tau_{0}}{2}X\cdot\varphi_{0}=\frac{3\|T\|}{\sqrt{7}}X\cdot\varphi_{0}.
 \end{equation}
Of course, the same result  occurs after applying our twistor equation for $\varphi_{0}$ and for some $s\neq   1/4$, see (\ref{t1s}) and the proof of Theorem \ref{general1}. Due to $\nabla^{c}$-parallelism of $\varphi_{0}$ we finally  conclude that
  \begin{eqnarray*}
 \nabla_{X}^{s}\varphi_{0}&=&\nabla^{c}_{X}\varphi_{0}+\frac{(4s-1)}{4}(X\lrcorner T)\cdot\varphi_{0}=\frac{(4s-1)}{4}(X\lrcorner T)\cdot\varphi_{0}\\
 &=&\frac{(4s-1)\tau_{0}}{8}X\cdot\varphi_{0}=\frac{3(4s-1)\|T\|}{4\sqrt{7}}X\cdot\varphi_{0}.
  \end{eqnarray*}
  Thus, for any $s\neq 0, 1/4$ it is  $\varphi_{0}\in\cal{K}^{s}(M^{7}, g)_{\zeta}$ with $\zeta:=\frac{(4s-1)\tau_{0}}{8}$ and moreover $\varphi_{0}\in\cal{K}(M^{7}, g)_{\kappa}$ with $\kappa=-\frac{\tau_{0}}{8}=-\frac{3\|T\|}{4\sqrt{7}}$. The opposite direction is based again in (\ref{xtg2}), which also occurs by the Riemannian Killing spinor equation $\nabla^{g}_{X}\varphi_{0}=-\frac{1}{4}(X\lrcorner T)\cdot\varphi_{0}=-\frac{3\|T\|}{4\sqrt{7}}X\cdot\varphi_{0}$.  Hence, using (\ref{nsc}) it follows that $\varphi_{0}$ is also a KsT for any $s\neq 0, 1/ 4$, and finally by  the twistor equation we get that $\varphi_{0}\in\Sigma_{-\sqrt{7}\|T\|}$.
  \end{proof}

 \subsection{An explicit example}
 Let us describe   an explicit example, namely the space $B^{7}=\SO_{5}/\SO_{3}^{\rm ir}$.  M.~Berger    proved that this is a space of positive sectional curvature.
       
    

 

Consider the space   $\Sym_{0}^{2}(\bb{R}^{3})$  of  $(3\times 3)$ symmetric traceless matrices; we identify   $\Sym_{0}^{2}(\bb{R}^{3})\cong \bb{R}^{5}$ by viewing   any vector $(x_{1}, \dots, x_{5})^{t}$ in $\bb{R}^{5}$ as a  real matrix      $A$ of the form
   \[
    A=\begin{bmatrix}
    \frac{x_{5}}{\sqrt{3}}+x_{1}  & x_{2} & x_{3} \\
    x_{2} &  \frac{x_{5}}{\sqrt{3}}-x_{1} & x_{4} \\
    x_{3} & x_{4} &  \frac{-2x_{5}}{\sqrt{3}}
    \end{bmatrix}\in \Sym_{0}^{2}(\bb{R}^{3}).
   \]
 The Lie group  $\SO_{3}$  acts on $\Sym_{0}^{2}(\bb{R}^{3})\cong\bb{R}^{5}$ by conjugation   $\iota(h)A=hAh^{t}$. This  defines the  unique 5-dimensional $\SO_{3}$-irreducible representation   and an  embedding of $\SO_{3}$ inside $\SO_{5}$, which we shall denote by $\SO_{3}^{\rm ir}:=\iota(\SO_{3})\subset\SO_{5}$.   For   the Lie algebra $\fr{so}(3)\subset\fr{so}(5)$ we fix the standard basis, i.e.  $\fr{so}(3)=\Span\{y_{1}:=E_{2, 3},  y_{2}:=-E_{1, 3},  y_{3}:=E_{1, 2}\}$  such that $[y_{1}, y_{2}]=y_{3}$, $[y_{2}, y_{3}]=y_{1}$ and $[y_{3}, y_{1}]=y_{2}$, where $E_{i, j}$ denote the   endomorphisms
mapping $e_{i}$ to $e_{j}$, $e_{j}$ to $-e_{i}$ and   everything else to zero.
The embedding $\fr{so}(3)_{\rm ir}\subset\fr{so}(5)$ is explicitly given by
      \[
y_{1}\mapsto  \iota_{*}(y_{1})=\sqrt{3}E_{1, 5}-E_{2, 5}+E_{3, 4}, \quad y_{2}\mapsto \iota_{*}(y_{2})=-\sqrt{3}E_{1, 4}-E_{2, 4}-E_{3, 5}, \quad y_{3}\mapsto \iota_{*}(y_{3})=2E_{2, 3}+E_{4, 5}.
  \]
  These matrices  have length equal to $\sqrt{5}$  and define an orthogonal basis of $\fr{so}(3)_{\rm ir}:=\iota_{*}(\fr{so}(3)\big)$  with respect to the scalar product $(A, B)=-(1/2)\tr AB$, i.e. $\fr{so}(3)_{\rm ir}=\Span\{\iota_{*}(y_{1}),  \iota_{*}(y_{2}),  \iota_{*}(y_{3})\}$. 
  
  Let  $\fr{so}(5)=\fr{so}(3)_{\rm ir}\oplus\fr{m}$ be a reductive decomposition and let us denote by $\langle \ , \ \rangle=( \ , \ )|_{\fr{m}\times\fr{m}}$ the normal metric induced by $( \ , \ ) : \fr{so}(5)\times\fr{so}(5)\to\bb{R}$.
    We identify  $\fr{m}\cong\bb{R}^{7}$ with the imaginary octonions ${\rm Im}(\bb{O})$ and construct an orthonormal basis of $\fr{m}$ such that $[e_{i}, e_{i+1}]$ be a multiple of $e_{i+3}$,  where the indices are permuted cyclically and translated modulo $7$.    For example, set   $e_{1}:=-E_{1, 2}$ and
  \[
  \begin{tabular}{l l l}
   $e_{2}:=-E_{1, 3}$, &  $e_{4}:=\displaystyle\frac{1}{\sqrt{5}}(E_{2, 3}-2E_{4, 5})$,  &  $e_{6}:=\displaystyle\frac{1}{2}(E_{1, 5}-\sqrt{3}E_{3, 4})$,  \\
  $e_{3}:=-\displaystyle\frac{1}{2}(E_{1, 4}-\sqrt{3}E_{3, 5})$,      &
  $e_{5}:=\displaystyle\frac{2}{\sqrt{5}}(E_{2, 5}+\frac{1}{4}E_{3, 4}+\frac{\sqrt{3}}{4}E_{1, 5})$, & $e_{7}:=\displaystyle\frac{2}{\sqrt{5}}(E_{2, 4}-\frac{1}{4}E_{3, 5}-\frac{\sqrt{3}}{4}E_{1, 4})$.
\end{tabular}
  \] 
Then  $[e_{i}, e_{i+1}]=ce_{i+3}$ with $c:=1/\sqrt{5}$ (see also \cite{Go} but be aware for another realization of $\fr{so}(3)_{\rm ir}\subset\fr{so}(5)$ and so another  basis of $\fr{m}$).  Since both $\SO_{3}^{\rm ir}$ and $\G_2$ preserving  the splitting $\bb{O}=\bb{R}\oplus{\rm Im}(\bb{O})$, the  $\G_2$-equivariant identification $\fr{m}\cong{\rm Im}(\bb{O})$ induces a 2-fold cross product an hence an invariant $\G_2$-structure on  $B^{7}$.  Thus, the image  of $\SO_{3}^{\rm ir}$ via the isotropy representation   lies inside $\G_2$.   
 The    isotropy representation  coincides with  the  unique 7-dimensional irreducible representation of $\SO_{3}$, which  is induced by the action of $\SO_{3}$ on the space $\Sym_{0}^{3}(\bb{R}^{3})\cong\bb{R}^{7}$ of trace-free 3-symmetric tensors on $\bb{R}^{3}$.
 
  Let us  describe    the 3-form $\omega$ associated to the $\G_2$-structure.  It is well-known that the space of  invariant 3-forms on $B^{7}$ contains the trivial summand with multiplicity one, i.e. $\bb{R}\subset\bigwedge^{3}(\fr{m})^{\fr{k}}$ with $\fr{k}:=\fr{so}(3)_{\rm ir}$, see \cite{Br1, Alex}.  
Since $B^{7}$ is (strongly) isotropy irreducible, Schur's lemma ensures that this corresponds to the torsion of the canonical connection $\nabla^{1}$  with respect to the fixed reductive decomposition $\fr{so}(5)=\fr{so}(3)_{\rm ir}\oplus\fr{m}$,  namely  
  \[
 T^{1}=\sum_{i<j<k}T^{1}(e_{i}, e_{j}, e_{k})e_{ijk}=-\frac{1}{\sqrt{5}}(e_{124}+e_{137}+e_{156}+e_{235}+e_{267}+e_{346}+e_{457}),  
 \]
 where  we write $e_{i_{1}\ldots i_{k}}$ for the wedge product $e_{i1}\wedge\cdots\wedge e_{ik}\in\bigwedge^{k}(\bb{R}^{7})^{*}$.  It is easy to see that $\|T^{1}\|^{2}=\frac{1}{3}\sum_{i<j}\langle T^{1}(e_{i}, e_{j}), T^{1}(e_{i}, e_{j})\rangle=\frac{21}{3}c^{2}=7/5$, hence let us  set $\omega:=-\sqrt{5}T^{1}$  such that $\|\omega\|^{2}=7$. Obviously   $\omega\in\Omega_{+}^{3}$ and its Hodge dual  is given by $\ast\omega=e_{1236}-e_{1257}-e_{1345}+e_{1467}+e_{2347}-e_{2456}-e_{3567}$. 
 
  Based on simple representation theory  we deduce  that $\omega$ induces a nearly parallel $\G_2$-structure. For example, the Hodge star operator $\ast$ allows us to identify $\bigwedge^{3}(\fr{m})^{\fr{k}}\cong \bigwedge^{4}(\fr{m})^{\fr{k}}$, as $\fr{k}$-modules. The differential of a 3-form $\bigwedge^{3}(\fr{m})^{\fr{k}}$ is again $\fr{k}$-invariant  and  the exterior differential $d :  \bigwedge^{3}(\fr{m})^{\fr{k}}\to \bigwedge^{4}(\fr{m})^{\fr{k}}$ is an equivariant map.  Hence $d\omega$ must be a multiple of the trivial summand in $ \bigwedge^{4}(\fr{m})^{\fr{k}}$ which means  that $d\omega$ and $\ast\omega$ must be proportional. 
    In particular, we see that $dT^{1}=-\frac{6}{5}(e_{1236}-e_{1257}-e_{1345}+e_{1467}+e_{2347}-e_{2456}-e_{3567})$      and    since   $d\omega=-\sqrt{5}dT^{1}$ it follows that $d\omega=\frac{6}{\sqrt{5}} \ast\omega$.  Therefore $d\omega\neq 0$ (which is equivalent to say tha $\omega$ is not parallel) and moreover  $d\ast\omega=0$, i.e. $\omega$ is co-calibrated.   The co-differential vanishes $\delta\omega=0$ (since $\delta T^{1}=0$) and hence the equation $d\omega=-8\kappa\ast\omega$  is equivalent to the Killing spinor equation $\nabla^{g}_{X}\varphi=\kappa X\cdot\varphi$, see \cite[Prop.~3.12]{FKMS}.   It follows that the coset $(\SO_{5}/\SO^{\rm ir}_{3}, \langle \ , \ \rangle, \omega)$ is a homogeneous nearly parallel $\G_2$-manifold with Killing number $\kappa:=-\frac{3}{4\sqrt{5}}$ and constant scalar curvature $\Sca^{g}=4n(n-1)\kappa^{2}=189/10$.  The corresponding Killing spinor field $\varphi_{0}$ is $\nabla^{1}$-parallel \cite[Thm.~5.6]{FrIv}, in fact $\varphi_{0}$ generates the  space of all $\nabla^{1}$-parallel spinors. 
   Because  $\SO_{5}/\SO^{\rm ir}_{3}$ is normal homogeneous, it is $\fr{g}=\tilde{\fr{g}}=\fr{m}+[\fr{m}, \fr{m}]$ and $\varphi_{0}$ necessarily corresponds to a constant map $\varphi_{0} : G\to\Delta_{7}$ \cite[Thm.~4.2]{Agr03}.   Thus, any $X\in\fr{m}$ satisfies the  equation $\nabla^{g}_{X}\varphi = \nabla^{1}_{X}\varphi+\widetilde{\Lambda^{g}}(X)\varphi= \widetilde{\Lambda^{g}}(X)\varphi_{0}=-\frac{3}{4\sqrt{5}} X\cdot \varphi_{0}$,  where  one describes the lift $\widetilde{\Lambda^{g}} : \fr{m}\to\fr{spin}(\fr{m})$  by using  the  Nomizu map $\Lambda^{g}(X)Y=(1/2)[X, Y]_{\fr{m}}$  of the Levi-Civita connection $\nabla^{g}$  and applying the rule $\fr{so}(7)\ni E_{i, j}\mapsto (e_{i}\cdot e_{j}/2)\in \fr{spin}(7)$. For the endomorphism $\Lambda^{g} : \fr{m}\to\fr{so}(7)$ we compute
\begin{gather*}
\begin{aligned}
& \Lambda^{g}(e_{1})=\frac{c}{2}(E_{2, 4}+E_{3, 7}+E_{5, 6}), \quad &   \Lambda^{g}(e_{5})= -\frac{c}{2}(E_{1, 6}-E_{2, 3}+E_{4, 7}), & \\
 &\Lambda^{g}(e_{2})=\frac{c}{2}(-E_{1, 4}+E_{3, 5}+E_{6, 7}), \quad &     \Lambda^{g}(e_{6})= \frac{c}{2}(E_{1, 5}-E_{2, 7}+E_{3, 4}), & \\ 
&\Lambda^{g}(e_{3})=-\frac{c}{2}(E_{1, 7}+E_{2, 5}-E_{4, 6}), \quad &    \Lambda^{g}(e_{7})= \frac{c}{2}(E_{1, 3}-E_{2, 6}+E_{4, 5}). & \\
& \Lambda^{g}(e_{4})=\frac{c}{2}(E_{1, 2}-E_{3, 6}+E_{5, 7}),    
\end{aligned}
 \end{gather*}
Then, the relation $\widetilde{\Lambda^{\al}}=(1-\al)\widetilde{\Lambda^{g}}$ shows that $\varphi_{0}$ is  also a non-trivial  Killing spinor with torsion with Killing number $\zeta=\frac{-3(1-\al)}{4\sqrt{5}}$, for any $\al\neq 0, 1$.   It remains to examine the eigenvalues of  the $T^{1}$-action on the spinor bundle $\Sigma=\SO_{5}\times_{\rho}\Delta_{7}$. 
  The real Clifford algebra  $\Cl(\bb{R}^{7})$ coincides with $M_{8}(\bb{R})\oplus M_{8}(\bb{R})$ and  the spin representation $\Delta_{7}$ is a real representation.   The Clifford representation attains  the matrix realization given in \cite[p.~261]{FKMS} or \cite[p.~96]{Baum} and   then, as an endomorphism of $\Delta_{7}:=\bb{R}^{8}$, the torsion form $T^{1}$    reads  
 \[
  T^{1}=-\frac{1}{\sqrt{5}}\left(  \begin{tabular}{c c c c c c c c}
 0 & -1 & -1  &1 & -1 & -1 & 1 & 1 \\
 -1 & 0 & 1 & -1 & 1 &  1 & -1 & -1 \\
 -1 & 1 & 0 & -1 & 1 & 1 & -1 & -1\\
 1 & -1 & -1 & 0 & -1 & -1 & 1& 1\\
 -1 & 1 & 1 & -1 & 0 &  1 & -1 & -1 \\
 -1 & 1 & 1 & -1 & 1 & 0 & -1& -1\\
1 & -1 & -1 & 1 & -1 & -1 & 0 & 1\\
 1 & -1 & -1 & 1 & -1 & -1 & 1 & 0 \\
 \end{tabular}\right)\in\Ed(\Delta_{7}).
 \]
We see that there is  a unique negative eigenvalue with multiplicity one, namely   $\gamma:=-7/\sqrt{5}$  and this is  exactly the result  described in Theorem \ref{general1} or \ref{nG2}:
 \[
 \gamma:=\frac{4n\kappa}{3} \ \Rightarrow \ \gamma=-\frac{4 \cdot 7 \cdot 3}{3\cdot 4\cdot \sqrt{5}}=-7/\sqrt{5}=-\sqrt{7}\|T^{1}\|.
 \]
We finish   with a  remark about the first eigenvalue of the cubic Dirac operator $\slashed{D}\equiv D^{1/3}$. According to  \cite[Thm.~3.3]{Agr03}, the square of $\slashed{D}$ is given by $\slashed{D}^{2}(\varphi)=\Omega_{\fr{so}(5)}(\varphi)+\frac{1}{8}\Sca^{1/3}\varphi+\frac{3}{4}\|T^{1/3}\|^{2}\varphi$.  A short  computation shows that $\Sca^{1/3}=560/30$, hence $\slashed{D}^{2}=\Omega_{\fr{so}(5)}+\frac{49}{20}$. In this way we overlap the general formula of the Casimir operator on   a nearly parallel $\G_2$-manifold,  described by I. Agricola and Th. Friedrich in  \cite[p.~200]{Cas}; $\Omega=\slashed{D}^{2}-\frac{49}{144}\tau_{0}^{2}$  and for  $\tau_{0}=6/\sqrt{5}$ it yields the desired result.  Using this relation,  one concludes  that    $(\lambda_{1}^{1/3})^{2}\geq \frac{49}{20}$, with the equality appearing if and only if the  $\varphi$ coincides with the spinor field $\varphi_{0}$. The equality case has been already indicated in   Section \ref{good1}; $\varphi_{0}$ is an eigenspinor of $\slashed{D}$  with eigenvalue $-\frac{1}{2}\gamma$, hence $(\lambda_{1}^{1/3})^{2}=\frac{1}{4}\gamma^{2}$ and for $\gamma=-7/\sqrt{5}$, our claim follows.  The same states Proposition \ref{afterklik} which we shall describe in the final section; the two estimates ought to coincide $\be_{\rm tw}(\gamma)=\be_{\rm univ}(\gamma)=\frac{7}{54}\Sca^{g}=\frac{7\cdot189}{54\cdot 10}=\frac{49}{20}$. 
 
 \medskip 
\noindent{\bf Remarks:} \textnormal{B.~Alexandrov and  U.~Semmelmann proved in \cite[Lem.~7.1]{Alex} that  on a 7-dimensional    homogenous  naturally reductive nearly parallel $\G_2$-manifold $(M=G/K, g, \omega)$, the characteristic connection $\nabla^{c}$ coincides  with the canonical connection $\nabla^{1}$.   Moreover, they showed that if $T^{1}=-\frac{\tau_{0}}{6}\omega$ holds for a stable 3-form $\omega$ and some $\tau_{0}=\text{constant}\neq 0$, then it must be $d\omega=\tau_{0}\ast\omega$ and $\Sca^{g}=\frac{63}{20\xi^{2}}$,  where  $M=G/K$ becomes standard up to the factor $\xi^{2}$, i.e. we assume that $g$ is the normal metric  $-\xi^{2}B_{\SO_{5}}|_{\fr{m}\times\fr{m}}$.   For $B^{7}$ we proved that $\tau_{0}=6/\sqrt{5}$.    On the other hand, the used scalar product  $( \ , \ )$ is a multiple of $-B_{\SO_{5}}$, in particular  $-B_{\SO_{5}}=6( \ , \ )$ and since our normal metric $\langle \ , \ \rangle$ is given by the restriction $( \ , \ )|_{\fr{m}}$, it follows that $\xi^{2}=1/6$. Thus $\Sca^{g}=189/10$, which shows how our computations agree with \cite[Lem.~7.1]{Alex}.}

\section{Geometric Constraints}\label{app}
\subsection{Einstein and $\nabla^{c}$-Einstein structures}
We shall now describe   the geometric constraints that imposes the existence of a $\nabla^{c}$-parallel KsT  (with respect to $\nabla^{s}=\nabla^{g}+2sT$).  We use the same notation with  Section  \ref{good1}, i.e. we assume that the triple $(M^{n}, g, T)$ is endowed with the characteristic connection   $\nabla^{c}=\nabla^{g}+\frac{1}{2}T$, such that $\nabla^{c}T=0$.  Let us recall that $(M^{n}, g, T)$  is said to be  $\nabla^{c}$-Einstein with parallel skew-torsion, if  it satisfies the equation $\Ric^{c}=\frac{\Sca^{c}}{n} g$  (and $\nabla^{c}T=0$)\cite{AFer}. A special  case is when  $\Ric^{c}=0$ identically; then $(M^{n}, g, T)$ is called $\Ric^{c}$-flat. For convenience, we shall henceforth speak for  a {\it strict} $\nabla^{c}$-Einstein manifold if    $(M^{n}, g, T)$ is $\nabla^{c}$-Einstein but not  $\Ric^{c}$-flat, i.e. $\Ric^{c}=\frac{\Sca^{c}}{n}g\neq 0$.

          \begin{prop}\label{bravo}
  Assume that $\nabla^{c}T=0$ and that $(M^{n}, g, T)$ is complete and admits  a $\nabla^{c}$-parallel spinor  $0\neq \varphi\in\Sigma_{\gamma}$ $(\bb{R}\ni\gamma\neq 0)$ lying in the kernel $\ke(P^{s})$ for some  $s\neq   1/4$.   Then, for any $s\in \bb{R}$  the following   hold
  \begin{eqnarray*}
      \Ric^{s}(X)\cdot\varphi&=&\frac{6\gamma^{2}}{n^{2}}\Big[\frac{6(n-1)(1-4s)^{2}+96s(1-4s)+16s(3-4s)(n-3)}{16}\Big]X\cdot\varphi, \\
  \Sca^{s}\varphi&=& \frac{6\gamma^{2}}{n}\Big[\frac{6(n-1)(1-4s)^{2}+96s(1-4s)+16s(3-4s)(n-3)}{16}\Big]\varphi.
    \end{eqnarray*}
In particular, 
  
  (a) $(M^{n}, g)$    is a compact Einstein  manifold with constant positive scalar curvature $\Sca^{g}=\frac{9(n-1)\gamma^{2}}{4n}$.
  
  (b)   For any $n>3$, $(M^{n}, g, T)$     is a strict $\nabla^{c}$-Einstein manifold with parallel torsion and constant scalar curvature $\Sca^{c}=\frac{3(n-3)\gamma^{2}}{n}$.   For $n=3$,  $(M^{3}, g, T)$  is $\Ric^{c}$-flat. 
   
   (c)  $(M^{n}, g, T)$ is $\nabla^{s}$-Einstein (with non-parallel torsion) for any $s\in\bb{R}\backslash\{0, 1/4\}$ i.e. $\Ric^{s}=\frac{\Sca^{s}}{n}g$.
   \end{prop}
   Before proceed with a proof of Proposition \ref{bravo}, let us remark that due to Theorem \ref{general1},   one can replace the assumptions $\nabla^{c}\varphi=0$ and  $\varphi\in\ke(P^{s}\big|_{\Sigma_{\gamma}})$   for some $s\neq 1/4$, with $\nabla^{c}\varphi=0$ and  $\varphi\in\cal{K}^{s}(M, g)_{\zeta}$ for some $s\neq 0, 1/4$, where the Killing number is defined by $\zeta:=3(1-4s)\gamma/4n$ for some  $0\neq\gamma\in\Spec(T)$, or  $\nabla^{c}\varphi=0$ and $\varphi\in\cal{K}(M, g)_{\kappa}$ with Killing number  $\kappa=3\gamma/4n$.  Thus, since $\varphi$ is necessarily a  Riemannian Killing spinor with real Killing spinor ${3\gamma}/{4n}$, we immediately conclude that (see for example \cite[p.~30]{Baum}):
\begin{equation}\label{rig}
\Ric^{g}(X)\cdot\varphi=4\kappa^{2}(n-1)X\cdot\varphi=\frac{9(n-1)\gamma^{2}}{4n^{2}}X\cdot\varphi.
\end{equation}
  In particular,  $(M^{n}, g)$ must be  Einstein   with positive scalar curvature given by $\Sca^{g}=4\kappa^{2}n(n-1)=\frac{9(n-1)\gamma^{2}}{4n}$.    After that,   Myers's theorem  ensures that $M$ is compact (here we also use that $M$ is complete). 
 
   Next we  present  a different proof of the fact that $g$ is Einstein,   without using arguments of the type that  such a  spinor  must be a real Killing spinor.   Our computations take  place in the spinor bundle and  it is very useful to  start   with a proof of  assertion $(b)$ i.e.  we provide  first the existence of a   $\nabla^{c}$-Einstein structure (and its explicit form), and then we use this fact to describe the  original Einstein condition.   Notice that  this approach differs  from  the   way that $\nabla^{c}$-Einstein structures have been  traditionally examined in dimensions 6 and 7, see \cite{FrIv} and compare with our Examples \ref{exnk} and \ref{exg2}   below. Observe also that  one is not allowed to apply Corollary \ref{lemkst}  to compute the Ricci tensor $\Ric^{c}$, since  $\varphi$ is a Killing spinor with torsion only for $s\neq 0, 1/4$ (recall that we do not view $\nabla^{s}$-parallel spinors as KsT), and this    asserts to our approach a much more special character. Finally, we mention that the $\nabla^{c}$-Einstein condition plays also a crucial role  for the proof of   the last claim, since it encodes the action  $(X\lrcorner \sigma_{T})\cdot\varphi=\frac{3\gamma^{2}(n-3)}{n^{2}}X\cdot\varphi$ (recall that $\varphi$ is always parallel with respect to $\nabla^{c}=\nabla^{g}+\frac{1}{2}T$).  

    \begin{proof}
 \noindent{\bf (b)}   According to Theorem \ref{KNOWN1}, the Ricci tensor $\Ric^{c}$ is given by  $\Ric^{c}(X)\cdot\varphi =\frac{1}{2}(X\lrcorner dT)\cdot\varphi =(X\lrcorner \sigma_{T})\cdot\varphi$.  On the other hand, for any vector field $X$ it holds that \cite[p.~325]{ABK}
\[
-2(X\lrcorner \sigma_{T})=\frac{1}{2}(T^{2}\cdot X-X\cdot T^{2})=(X\lrcorner T)\cdot T-T\cdot (X\lrcorner T).
\]
 To give a few hints for this very useful rule, let us   write first $-X\lrcorner \sigma_{T}=\frac{1}{2}(X\cdot\sigma_{T}-\sigma_{T}\cdot X)$.  Replacing in the right-hand side the 4-form $\sigma_{T}$ by $\sigma_{T}=\frac{1}{2}(\|T\|^{2}-T^{2})$,   the first relation becomes obvious.   Notice now that $X\cdot T+T\cdot X=-2(X\lrcorner T)$. Then,   combining  the resulting formulae of the Clifford multiplication with $T$, once from the left and once from the right, we see that $(T^{2}\cdot X-X\cdot T^{2})=2\Big((X\lrcorner T)\cdot T-T\cdot (X\lrcorner T)\Big)$.  Hence,  the action of the Ricci endomorphism $\Ric^{c}(X)$ on $\nabla^{c}$-parallel spinors  reads
\[
 \Ric^{c}(X)\cdot\varphi =-\frac{1}{2}\Big[(X\lrcorner T)\cdot T-T\cdot (X\lrcorner T)\Big]\cdot\varphi.  \quad {\small{(\clubsuit)}}
 \]
 Now, our assumption tell us that $\varphi\in\ke(P^{s}\big|_{\Sigma_{\gamma}})$ for some $s\neq  1/4$ and $\gamma\neq 0$; thus,   (\ref{t1s}) needs to hold and this is the key information  that the twistor equation carries (of course, due to Theorem \ref{general1}, the same holds if we start with a $\nabla^{c}$-parallel real Killing spinor $\varphi$ with $\kappa=3\gamma/4n$, see also Proposition \ref{klik}).  Based on this formula  and using   $T\cdot\varphi=\gamma\varphi$, one easily computes
 \begin{eqnarray*}
 \Big[(X\lrcorner T)\cdot T-T\cdot (X\lrcorner T)\Big]\cdot\varphi&=&\gamma(X\lrcorner T)\cdot\varphi+\frac{3\gamma}{n}T\cdot X\cdot \varphi\\
 &=&-\frac{3\gamma^{2}}{n}X\cdot\varphi+\frac{3\gamma}{n}\Big(-X\cdot T\cdot\varphi-2(X\lrcorner T)\cdot\varphi\Big)\\
 &=&-\frac{3\gamma^{2}}{n}X\cdot\varphi-\frac{3\gamma^{2}}{n}X\cdot\varphi-\frac{6\gamma}{n}(X\lrcorner T)\cdot\varphi\\
 &=&-\frac{6\gamma^{2}}{n}X\cdot\varphi+\frac{18\gamma^{2}}{n^{2}}X\cdot\varphi \\
 &=&\frac{6\gamma^{2}}{n}\Big[\frac{3}{n}-1\Big]X\cdot\varphi.
 \end{eqnarray*}
Thus, any $X\in\Gamma(TM)$ satisfies the equation 
\begin{equation}\label{ricapl}
 \Ric^{c}(X)\cdot\varphi =\frac{3(n-3)\gamma^{2}}{n^{2}}X\cdot\varphi.
 \end{equation}
Now,  $\varphi$ cannot have zeros, since for example by Theorem \ref{general1} it  is also a non-trivial KsT with respect to the family $\nabla^{s}$ and so parallel with respect to the connection $\nabla:=\nabla^{s}_{X}-\zeta X$, where $\zeta=\frac{3(1-4s)\gamma}{4n}\neq 0$ (and the same occurs for instance since $\varphi$ is $\nabla^{c}$-parallel).  It follows that  for $n>3$, the triple  $(M^{n}, g, T)$ is a strict $\nabla^{c}$-Einstein manifold with constant  positive scalar curvature $\Sca^{c}=\frac{3(n-3)\gamma^{2}}{n}$.  For $n=3$ we get $\Ric^{c}(X)\cdot\varphi=0$ for any $X$, so $(M^{3}, g, T)$ is $\Ric^{c}$-flat.   
 
  \smallskip
\noindent{\bf Alternative proof of (a)}
In the proof of \cite[Thm.~A.2]{ABK} (see page 325), for  $s=\frac{(n-1)}{4(n-3)}$ and for the family $\nabla^{s}_{X}\varphi=\nabla^{c}+\lambda(X\lrcorner T)\cdot\varphi$  with $\lambda:=\frac{1}{2(n-3)}$,  the following formula was presented for the curvature tensor $R^{s} : \Lambda^{2}(TM)\to\Ed(\Sigma)$ associated to $\nabla^{s}$:
\begin{eqnarray*}
\sum_{i}e_{i}\cdot R^{s}(X, e_{i})\varphi&=&\sum_{i}e_{i}\cdot R^{c}(X, e_{i})\varphi-6\lambda^{2}(X\lrcorner \sigma_{T})\cdot\varphi\nonumber\\
&&\quad\quad\quad\quad-(2\lambda^{2}+\lambda)\sum_{i}T(X, e_{i})\cdot (e_{i}\lrcorner T)\cdot\varphi,
\end{eqnarray*}
 where $\{e_{1}, \ldots, e_{n}\}$ is an orthonormal frame of $M$. Exactly the same  formula holds   for $\lambda=\frac{4s-1}{4}$, i.e. our family $\nabla^{s}$ and for any $s\in\bb{R}$.  In particular, for $s=\frac{(n-1)}{4(n-3)}$ the quantity  $\frac{4s-1}{4}$ is nothing than 
 the fixed $\lambda:=\frac{1}{2(n-3)}$.  Therefore, in the previous formula one can replace $\lambda$ by $\frac{4s-1}{4}$  and let  $s$ running in $\bb{R}$. Then, for $s=0$ we get a  useful  expression between the curvature tensors $R^{c}$ and $R^{g}$ associated to the characteristic connection and the Riemannian connection, respectively:
 \begin{equation}\label{conv}
 \sum_{i}e_{i}\cdot R^{g}(X, e_{i})\varphi=\sum_{i}e_{i}\cdot R^{c}(X, e_{i})\varphi-\frac{6}{16}(X\lrcorner \sigma_{T})\cdot\varphi+\frac{1}{8}\sum_{i}T(X, e_{i})\cdot (e_{i}\lrcorner T)\cdot\varphi.
 \end{equation}
This formula holds for any $\varphi\in \Gamma(\Sigma)$ and $X\in\Gamma(TM)$ and plays  a crucial role in what follows.

\noindent So, let us emphasize on our case.  By assumption $\varphi\in\Sigma_{\gamma}$ is $\nabla^{c}$-parallel, hence the first term   in the right hand side vanishes, i.e. $\sum_{i}e_{i}\cdot R^{c}(X, e_{i})\varphi\equiv 0$ identically.  Moreover, we  know from (b) that  $(M^{n}, g, T)$ is $\nabla^{c}$-Einstein with respect to the same metric $g$, and this determines the  second term: $(X\lrcorner\sigma_{T})\cdot\varphi=\Ric^{c}(X)\cdot\varphi=-\frac{3(3-n)\gamma^{2}}{n^{2}}X\cdot\varphi$.  Finally, for the computation of the third term we take advantage of the fact that $\varphi\in\ke(P^{s}\big|_{\Sigma_{\gamma}})$ for some $s\neq  1/4$ (which is equivalent to (\ref{t1s})).  Applying successively  this relation, yields
 \begin{eqnarray*}
 \sum_{i}T(X, e_{i})\cdot (e_{i}\lrcorner T)\cdot\varphi&=&-\frac{3\gamma}{n}\sum_{i}T(X, e_{i})\cdot e_{i}\cdot\varphi \nonumber\\
 &=&\frac{6\gamma}{n}(X\lrcorner T)\cdot\varphi=-\frac{18\gamma^{2}}{n^{2}}X\cdot\varphi.
  \end{eqnarray*}
Thus we obtain altogether: 
\[
 \sum_{i}e_{i}\cdot R^{g}(X, e_{i})\varphi=-\frac{9(n-1)\gamma^{2}}{8n^{2}}X\cdot\varphi,
\]
  for any  $X\in\Gamma(TM)$. Observing that $\sum_{i}e_{i}\cdot R^{g}(X, e_{i})\varphi=-\frac{1}{2}\Ric^{g}(X)\cdot\varphi$, see \cite[p.~15]{Baum}, one can easily finish the proof of part $(a)$.  
  
  \noindent{\bf (c)} Now,  the last part and the stated formulas for  $\Ric^{s}, \Sca^{s}$  is an immediate consequence of   Corollary \ref{lemkst} in combination with Proposition \ref{klik} and   relation  $(X\lrcorner \sigma_{T})\cdot\varphi=\frac{3\gamma^{2}(n-3)}{n^{2}}X\cdot\varphi$, which  still makes sense. We mention once more  that we  apply Corollary \ref{lemkst} for   $s\neq 0, 1/4$. However, the stated  formulas of $\Ric^{s}, \Sca^{s}$ produce the right results for any $s\in\bb{R}$ (even for $s=0, 1/4$). This completes the proof.
        \end{proof}
\begin{example}\label{exnk}\textnormal{(see also \cite[Prop.~10.4]{FrIv})}
\textnormal{Consider a nearly K\"ahler manifold $(M^{6}, g, J)$. Recall that there exist two $\nabla^{c}$-parallel spinors $\varphi^{\pm}$ with  $\gamma=\pm 2\|T\|$ which are both TsT for some $s\neq   1/4$. Hence, due to Proposition  \ref{bravo} we conclude that
\[
 \Ric^{s}(X)\cdot\varphi^{\pm}=\frac{(5-16s^{2})}{4}\|T\|^{2} X\cdot\varphi^{\pm}=\frac{(5-16s^{2})}{2}\tau_{0}X\cdot\varphi^{\pm}, \quad \forall \ s\in\bb{R}, 
 \]
 in particular
  \begin{eqnarray*}
    \Ric^{c}(X)\cdot\varphi^{\pm}&=&\frac{3(n-3)\gamma^{2}}{n^{2}}X\cdot\varphi^{\pm} \  \Rightarrow \  \Ric^{c}(X)\cdot\varphi^{\pm}=\|T\|^{2}X\cdot\varphi^{\pm}=2\tau_{0}X\cdot\varphi^{\pm} ,\\
    \Ric^{g}(X)\cdot\varphi^{\pm}&=&\frac{9(n-1)\gamma^{2}}{4n^{2}}X\cdot\varphi^{\pm} \  \Rightarrow \   \Ric^{g}(X)\cdot\varphi^{\pm}  =\frac{5}{4}\|T\|^{2}\cdot\varphi^{\pm}=\frac{5}{2}\tau_{0}X\cdot\varphi^{\pm}.
    \end{eqnarray*}
 On the other hand, relation    (\ref{wrong}) is still available for a straightforward  computation of  the Ricci tensor $\Ric^{c}$. Indeed, a direct computation shows that
  \begin{eqnarray*}
 \Big[(X\lrcorner T)\cdot T-T\cdot (X\lrcorner T)\Big]\cdot\varphi^{\pm}&=&\pm 2\|T\|(X\lrcorner T)\cdot\varphi^{\pm}\pm \|T\|T\cdot X\cdot \varphi^{\pm}\\
 &=&-2\|T\|^{2}X\cdot\varphi^{\pm}\pm\|T\|\Big(-2(X\lrcorner T)\cdot\varphi^{\pm}-X\cdot T \cdot\varphi^{\pm}\Big)\\
 &=&-2\|T\|^{2}X\cdot\varphi^{\pm}\pm\|T\|\Big(\pm 2\|T\|X\cdot\varphi^{\pm}\mp 2\|T\|X\cdot\varphi^{\pm}\Big)\\
 &=&-2\|T\|^{2}X\cdot\varphi^{\pm},
 \end{eqnarray*}
 and the result now is an immediate consequence of ${\small{(\clubsuit)}}$.  In this spinorial way and for $s=0, 1/4$, we overlap the results of  \cite[Prop.~10.4]{FrIv}. Mention however that our method is different, i.e. we are  {\it not}  based on the Einstein property of $g$ and    the relation $\Ric^{c}=\Ric^{g}-\frac{1}{4}S$. Finally, the type  $\Ric^{s}=\Ric^{g}-4s^{2}S$  produces the same results for all $s\in\bb{R}$, since $S:=2\tau_{0}\Id$, see  \cite[Prop.~10.4]{FrIv}. } 
 
\end{example}

 \begin{example}\label{exg2}\textnormal{(see also \cite[Thm.~5.1, Ex.~5.2]{FrIv})}
    \textnormal{Consider a    nearly parallel $\G_2$-manifold $(M^{7}, g, \omega)$. Recall that  there is a unique $\nabla^{c}$-parallel spinor field $\varphi_{0}$  with $\gamma=-\sqrt{7}\|T\|$. In a similar way with Example \ref{exnk}   and due to  Proposition \ref{bravo}, one gets that
       \[
\Ric^{s}(X)\cdot\varphi_{0}=\frac{6(9-16s^{2})}{28}\|T\|^{2}X\cdot\varphi_{0}=\frac{(9-16s^{2})}{24}\tau_{0}^{2}X\cdot\varphi_{0}, \quad \forall \ s\in\bb{R},
\]
     in particular
       \[
    \Ric^{c}(X)\cdot\varphi_{0}=\frac{12}{7}\|T\|^{2}X\cdot\varphi_{0}=\frac{\tau_{0}^{2}}{3}X\cdot\varphi_{0},\quad 
    \Ric^{g}(X)\cdot\varphi_{0}=\frac{27}{14}\|T\|^{2}X\cdot\varphi_{0}=\frac{3\tau_{0}^{2}}{8}X\cdot\varphi_{0}.
    \]
     In a line with nearly K\"ahler manifolds in dimension 6, relation (\ref{xtg2}) gives us the ability to compute $\Ric^{c}$ in a direct way.  Indeed, due to (\ref{teigen}) and (\ref{xtg2}), a little computation shows that
\[
 \Big[(X\lrcorner T)\cdot T-T\cdot (X\lrcorner T)\Big]\cdot\varphi_{0}=-\frac{24}{7}\|T\|^{2}X\cdot\varphi_{0},
\]
and the result follows by $(\clubsuit)$. Notice finally that one can reproduce all these results, using the formula $\Ric^{s}=\Ric^{g}-4s^{2}S$, see \cite[Ex.~5.2]{FrIv} for the tensor $S:=\frac{\tau_{0}^{2}}{6}\Id$. }
    \end{example}

\subsection{The 3-dimensional case}\label{friedrich}
 In      \cite[Ex.~7.2]{AF}  it is shortly explained that   the {\it unique} 3-dimensional  compact manifold carrying parallel spinors with respect to a metric connection with non-trivial  skew-torsion, is a space form, namely the round 3-sphere $\Ss^{3}=\Spin_{4}/\Spin_{3}$    endowed with the  canonical spin structure  $(P, \Lambda):= (\Spin_{3}\times\Spin_{3}, \Id _{\Spin_{3}}\times\lambda)$ where $\lambda : \Spin_{3}\to\SO_{3}$ is the double covering, 
 the  canonical metric $g:=g_{\rm can}$ of constant sectional  curvature 1   and  finally the volume form $T={\rm Vol}_{\Ss^{3}}$. 
 But let us explain how this can fit with our results and what new we can say. The spinor bundle $\Sigma:=P\times_{\SU_{2}}\Delta_{3}\to\Ss^{3}$ is trivialized through either the $-\frac{1}{2}$- or $\frac{1}{2}$-Killing spinors and   $T$ acts as the identity   operator  (and as a  scalar operator in the case that $T=f{\rm Vol}_{\Ss^{3}}$, for some constant $f$).  
By  \cite{AF} we know that the $\epsilon$-Killing spinors are   parallel with respect to the metric connections induced by the Killing spinor equation: $   \nabla^{\epsilon}_{X}\varphi=\nabla^{g}_{X}\varphi+\epsilon(X\lrcorner T)\cdot\varphi=\nabla^{g}_{X}\varphi- \epsilon X\cdot T\cdot \varphi=\nabla^{g}_{X}\varphi- \epsilon X\cdot \varphi=0$,   where  $\epsilon\in\{\pm\frac{1}{2}\}$.   For convenience, we set $T^{\epsilon}:=2T$ and rewrite  
\begin{eqnarray*}
\nabla^{\epsilon}_{X}\varphi&=&\nabla^{g}_{X}\varphi+\frac{\epsilon}{2}(X\lrcorner T^{\epsilon})\cdot\varphi=\nabla^{g}_{X}\varphi\pm \frac{1}{4}(X\lrcorner T^{\epsilon})\cdot\varphi=
 \nabla^{g}_{X}\varphi\mp \frac{1}{4}X\cdot  T^{\epsilon} \cdot\varphi=\nabla^{g}_{X}\varphi \mp\frac{1}{2} X\cdot  \varphi.
\end{eqnarray*}
 Now it is obvious  that $\nabla^{\epsilon}\equiv\nabla^{\pm 1/2}$ are metric connections with  skew-torsion 
 \[
 \pm T^{\epsilon}=\pm 2T=\pm 2 (e_{1}\wedge e_{2}\wedge e_{3}),\]
  such that $\nabla^{\epsilon}T^{\epsilon}=0$ (since  $\Ss^{3}$ is orientable and $T$ is the volume form). On $T\Ss^{3}$ one can write $\nabla^{\epsilon}=\nabla^{g}\pm \frac{1}{2}T^{\epsilon}$.
  Notice that since the sphere      $\Ss^{3}$ is diffeomorphic to the compact Lie group $\Spin_{3}\cong\SU_{2}\cong\Sp_{1}$ a characteristic (or canonical) connection  is not unique \cite{Olmos, Nagy, AFH}. Endowed with a bi-invariant metric and one of the $\pm$-canonical connections of Cartan-Schouten is flat and hence also $\Ric^{\pm}$-flat, see for instance \cite[Ex.~7.1]{AF}.    $\nabla^{\epsilon}$ is also flat, in particular   $\Ss^{3}$ is simply connected  and the flatness of $\nabla^{\epsilon}$ implies the triviality of the associated spinor bundle $\Sigma$ \cite[Lem.~1]{Bar2}. 
         
     \begin{lemma}
    $T^{\epsilon}(X, Y)\cdot\varphi=-(X\cdot Y-Y\cdot X)\cdot\varphi$, for any vector field $X, Y$ and spinor field $\varphi$. 
    \end{lemma}
  \begin{proof}
 {\bf 1st way.}  In \cite{Bar2}    appears  the following expression  (for $\epsilon=1/2$ and for a general sphere $\Ss^{n}$)
       \[
       \nabla^{\epsilon}_{X}\nabla^{\epsilon}_{Y}\varphi=(\nabla^{g}_{X}-\frac{1}{2}X)(\nabla^{g}_{Y}-\frac{1}{2}Y)\varphi=\nabla^{g}_{X}\nabla^{g}_{Y}\varphi-\frac{1}{2}Y\cdot\nabla^{g}_{X}\varphi-\frac{1}{2}X\cdot\nabla^{g}_{Y}\varphi+\frac{1}{4}X\cdot Y\cdot\varphi,
       \]
       where locally for the  Riemannian connection  one can assume   that $(\nabla^{g}X)(p)=(\nabla^{g}Y)(p)=0$, for some vector fields $X, Y\in\Gamma(T\Ss^{3})$ and $p\in \Ss^{3}$, see also \cite[p.~23]{Baum}.  However,  the same time we can write
       \begin{eqnarray*}
         \nabla^{\epsilon}_{X}\nabla^{\epsilon}_{Y}\varphi&=&\nabla^{\epsilon}_{X}(\nabla^{g}_{Y}\varphi-\frac{1}{2}Y\cdot\varphi)= \nabla^{\epsilon}_{X}\nabla^{g}_{Y}\varphi-\frac{1}{2}\nabla^{\epsilon}_{X}(Y\cdot\varphi)\\
        &=&         \nabla^{g}_{X}\nabla^{g}_{Y}\varphi-\frac{1}{2}X\cdot(\nabla^{g}_{Y}\varphi)-\frac{1}{2}(\nabla^{\epsilon}_{X}Y)\cdot\varphi-\frac{1}{2}Y\cdot \nabla^{g}_{X}\varphi+\frac{1}{4}Y\cdot X\cdot\varphi. \quad (\star)
               \end{eqnarray*}
A comparison  now with the previous  relation  shows that $(\nabla^{\epsilon}_{X}Y)\cdot\varphi=-\frac{1}{2}(X\cdot Y-Y\cdot X)\cdot\varphi$. Hence, after replacing $\nabla^{\epsilon}_{X}Y=\nabla^{g}_{X}Y+\frac{1}{2}T^{\epsilon}(X, Y)$,  we get     our assertion:
\[
T^{\epsilon}(X, Y)\cdot\varphi=-(X\cdot Y-Y\cdot X)\cdot\varphi.
\] 
     Similarly for the case of $\epsilon=-1/2$, i.e. the connection $\nabla^{-1/2}_{X}\varphi=\nabla^{g}_{X}\varphi+\frac{1}{2}X\cdot\varphi$ with torsion $-T^{\epsilon}$.
     
     \noindent {\bf 2nd way.}  (We again explain the case  $\epsilon=1/2$).  Notice that one can avoid the   assumption  $(\nabla^{g}X)(p)=(\nabla^{g}Y)(p)=0$, since  $(\star)$   itself is independent of this condition.  In the same way, we write
     \[
     \nabla^{\epsilon}_{Y}\nabla^{\epsilon}_{X}\varphi=\nabla^{g}_{Y}\nabla^{g}_{X}\varphi-\frac{1}{2}Y\cdot(\nabla^{g}_{X}\varphi)-\frac{1}{2}(\nabla^{\epsilon}_{Y}X)\cdot\varphi-\frac{1}{2}X\cdot \nabla^{g}_{Y}\varphi+\frac{1}{4}X\cdot Y\cdot\varphi 
     \]
    and as a consequence of the definition $R^{\epsilon}(X, Y)\varphi=   \nabla^{\epsilon}_{X}\nabla^{\epsilon}_{Y}\varphi-  \nabla^{\epsilon}_{Y}\nabla^{\epsilon}_{X}\varphi-\nabla^{\epsilon}_{[X, Y]}\varphi$, we see that (we give the expression for both $\epsilon\in\{\pm 1/2\}$)
     \begin{eqnarray}
     R^{\epsilon}(X, Y)\varphi
     &=&R^{g}(X, Y)\varphi\mp\frac{1}{2}\Big[\nabla^{\epsilon}_{X}Y-\nabla^{\epsilon}_{Y}X-[X, Y]\Big]\cdot\varphi+\frac{1}{4}(Y\cdot X-X\cdot Y)\cdot\varphi\nonumber\\
     &=&R^{g}(X, Y)\varphi-\frac{1}{2}T^{\epsilon}(X, Y)\varphi+\frac{1}{4}(Y\cdot X-X\cdot Y)\cdot\varphi.\label{S31}
     \end{eqnarray} 
           Here,  $\nabla^{\epsilon}_{X}Y- \nabla^{\epsilon}_{Y}X-[X, Y]$ equals  to $\pm T^{\epsilon}(X, Y)$, depending on $\epsilon\in\{\pm 1/2\}$.  Now,   based on the fact  that  $\Ss^{3}$ has constant sectional curvature 1  one computes $R^{g}(X, Y)\varphi=\frac{1}{4}(Y\cdot X-X\cdot Y)\cdot\varphi$, see \cite{Bar2} (or \cite[Thm.~8, p.~30]{Baum}). Combining this with  $R^{\epsilon}\equiv 0$, relation (\ref{S31}) gives  rise to the desired result.
            \end{proof}
        Let us identify the tangent space $T_{o}\Ss^{3}$ with $\fr{g}:=\fr{su}(3)=\fr{spin}(3)=\fr{so}(3)$ and denote by $\{e_{1}, e_{2}, e_{e}\}$   the left-invariant vector fields associated to a basis of $\fr{g}$. Then,  $T^{\epsilon}(e_{i}, e_{j})\cdot\varphi=-(e_{i}\cdot e_{j}-e_{j}\cdot e_{i})\cdot\varphi=-[e_{i}, e_{j}]\cdot\varphi$ holds for any $\varphi\in\Gamma(\Sigma)$. Because $\varphi$ can be written as a linear combination of  Killing spinors, $\varphi$ has no-zeros and thus $T^{\epsilon}(e_{i}, e_{j})=-[e_{i}, e_{j}]$.  Hence, this spinorial approach allows us to deduce that the torsion $\pm T^{\epsilon}$ coincides with the torsion of the $\pm 1$-Cartan-Schouten connections, as it should be due to the uniqueness of $\nabla^{\pm 1}$. 

Relation (\ref{S31}) has another remarkable application; it implies  that $(\Ss^{3}, g)$ is Einstein, without using    arguments of the type  that $\Ss^{3}$ is a space form in dimension 3, neither a manifold carrying  Killing spinors, nor  an isotropy irreducible homogeneous space. In fact, we do not  even use  the flatness of  $\nabla^{\epsilon}$, but only the fact that given a trivialization $\{\varphi_{j} : 1\leq j\leq 2^{[\frac{3}{2}]}\}$ of $\Sigma$ by $\epsilon$-Killing spinors, then the relation $\nabla^{\epsilon}\varphi_{j}=0$ needs to hold   $\forall \ j$ (in a similar way with nearly K\"ahler and nearly parallel $\G_2$-manifolds).            Indeed,   since $T^{\epsilon}\cdot\varphi_{j}=2T\cdot\varphi_{j}=2\varphi_{j}$,  applying Clifford multiplication  on (\ref{S31}) with respect to the orthonormal frame  $\{e_{1}, e_{2}, e_{3}\}$ and finally adding,   yields
  \begin{eqnarray*}
\sum_{i}e_{i}\cdot R^{\epsilon}(X, e_{i})\varphi_{j}&=&\sum_{i}e_{i}\cdot R^{g}(X, e_{i})\varphi_{j}-\frac{1}{2}\sum_{i}e_{i}\cdot T^{\epsilon}(X, e_{i})\cdot\varphi_{j} +\frac{1}{4}\sum_{i}e_{i}\cdot(e_{i}\cdot X-X\cdot e_{i})\cdot\varphi_{j}\\
&=&-\frac{1}{2}\Ric^{g}(X)\cdot\varphi_{j}-(X\lrcorner T^{\epsilon})\cdot\varphi_{j}-X\cdot\varphi_{j}=-\frac{1}{2}\Ric^{g}(X)\cdot\varphi_{j}+X\cdot T^{\epsilon}\cdot\varphi_{j}-X\cdot\varphi_{j}\\
&=&-\frac{1}{2}\Ric^{g}(X)\cdot\varphi_{j}+2X\cdot\varphi_{j}-X\cdot\varphi_{j}.
\end{eqnarray*}
           Because $\nabla^{\pm 1/2}\varphi_{j}=0$,  the left-hand side vanishes and the resulting formulae $\Ric^{g}(X)\cdot\varphi_{j}=2X\cdot\varphi_{j}$ shows that    $(\Ss^{3}, g_{\rm can})$ is   Einstein  with Einstein constant $2=\Sca^{g}/n=6/3$, see also \cite{Fr1980}.                   
 
 \smallskip
 Now, because $\nabla^{\epsilon}$ are flat, $(\Ss^{3}, g_{\rm can})$   is automatically $\Ric^{\epsilon}$-flat. This coincides with the statement of  Proposition \ref{bravo} (here we only allow $\epsilon=1/2$), namely  
  \begin{eqnarray*}
  \Ric^{g}(X)\cdot\varphi_{j}&=&\frac{9(n-1)\gamma^{2}}{4n^{2}}X\cdot\varphi_{j} \  \Rightarrow \  \Ric^{g}(X)\cdot\varphi_{j}=2X\cdot\varphi_{j},\\
  \Ric^{\epsilon}(X)\cdot\varphi_{j}&=&\frac{3(n-3)\gamma^{2}}{n^{2}}X\cdot\varphi_{j} \Rightarrow \ \Ric^{\epsilon}(X)\cdot\varphi_{j}=0.
  \end{eqnarray*}
More generally, $\Ric^{s}(X)\cdot\varphi_{j}=2(1-16s^{2})X\cdot\varphi_{j}$, i.e. 
\[
\Ric^{s}=2(1-16s^{2})\Id, \quad   \forall \ s\in\bb{R}.
\]
  Of course, and with the aim to apply Proposition \ref{bravo},  one has to consider  first the family
      \[
    \nabla^{\epsilon, s}_{X}\varphi:=\nabla^{g}_{X}\varphi+s(X\lrcorner T^{\epsilon})\cdot\varphi.
      \]
     For $s=\pm 1/4$ it induces the flat connections $\nabla^{\epsilon, \pm 1/4}=\nabla^{\epsilon}=\nabla^{\pm 1/2}$ and for $s=0$  it coincides with the spinorial Riemannian connection. Because the trivialization $\{\varphi_{j} : 1\leq j\leq 2^{[\frac{3}{2}]}\}$ of $\Sigma$ consists of $\epsilon$-Killing spinors which are $\nabla^{\epsilon}$-parallel, Theorem \ref{general1} states that
     \begin{theorem}\label{s3}
   There is a one-to-one correspondence between $\epsilon$-Killing spinors on    $(\Ss^{3}, g_{\rm can}, T^{\epsilon})$ and  Killing spinors with torsion with respect to the family $\nabla^{\epsilon, s}$ for any $s\neq 0, 1/4$, with Killing number $\zeta=\frac{1-4s}{2}$, i.e. $\nabla^{\epsilon, s}_{X}\varphi_{j}=\frac{1-4s}{2}X\cdot\varphi_{j}$, $\forall \ X\in\Gamma(T\Ss^{3})$.   In particular, a 3-dimensional compact spin  manifold   $(M^{3}, g, T)$ satisfying the assumptions of Proposition \ref{bravo}, is isometric to   $(\Ss^{3}, g_{\rm can}, T^{\epsilon})$.
           \end{theorem}
        \begin{proof}
Let us shortly present a direct proof. If $\nabla^{g}_{X}\varphi=\frac{1}{2}X\cdot\varphi$ for any $X\in\Gamma(T\Ss^{3})$, then
        \begin{eqnarray*}
        \nabla^{\epsilon, s}_{X}\varphi&=&\frac{1}{2}X\cdot\varphi+s(X\lrcorner T^{\epsilon})\cdot\varphi=\frac{1}{2}X\cdot\varphi+2s(X\lrcorner T)\cdot\varphi\\
        &=&\frac{1}{2}X\cdot\varphi-2sX\cdot T\cdot\varphi=\frac{1-4s}{2}X\cdot\varphi.
        \end{eqnarray*}
        Conversely, if $\varphi\in\cal{K}^{s}(\Ss^{3}, g)_{\zeta}$ with $\zeta=\frac{1-4s}{2}$ for some $s\neq 0, 1/4$, then
        $\nabla^{\epsilon, s}_{X}\varphi=\zeta X\cdot\varphi$ and thus 
        \begin{eqnarray*}
        \nabla^{g}_{X}\varphi=\zeta X\cdot\varphi-s(X\lrcorner T^{\epsilon})\cdot\varphi=\frac{1-4s}{2}X\cdot\varphi+2sX\cdot T\cdot\varphi=\frac{1}{2}X\cdot\varphi.
        \end{eqnarray*}
        \end{proof}

 \subsection{A partial classification}
 
  We deduce that on a triple $(M^{n}, g, T)$ with $\nabla^{c}T=0$, the existence of a spinor field $\varphi$ satisfying simultaneously the equations  
  \begin{equation}\label{cla}
  \nabla^{c}_{X}\varphi=0, \quad \nabla^{s}_{X}\varphi=\zeta X\cdot\varphi, \quad \forall \ X\in\Gamma(TM),
  \end{equation} 
    for some real numbers $s\neq  0, 1/4$, $\zeta\neq 0$, where $\nabla^{s}=\nabla^{g}+2sT$, imposes    much harder geometric restrictions than the original Killing spinor equation, namely:
\[
\begin{tabular}{l | c | l }
& {\sc Type of Killing spinors} &  {\sc  Geometric conclusions} \\
\thickline
   $\bold(\al)$ &  Killing spinors with Killing number  $\kappa\in\bb{R}\backslash\{0\}$  &     \  $\bullet$ \  $\Ric^{g}=4\kappa^{2}(n-1)g$, \  $\Sca^{g}=4\kappa^{2}n(n-1)$ \\
   \hline
   $\bold(\be)$ &  $\nabla^{c}$-parallel KsT  w.r.t. $\nabla^{s}=\nabla^{g}+2sT$  &   \  $\bullet$ \  $\varphi$ is a real Killing spinor: $T\cdot\varphi=\gamma\cdot\varphi\neq 0$ \\
      & with Killing number $\zeta=\frac{3(1-4s)\gamma}{n}\neq 0$  &   \ $\bullet$ \ $\Ric^{s}=\frac{\Sca^{s}}{n}g$ \ $\forall \ s\in\bb{R}$, in particular :\\
      & for some $\bb{R}\ni\gamma\neq0$, $\bb{R}\ni s\neq 0, 1/4$  & \  $-$ \ $\Ric^{g}=\frac{9(n-1)\gamma^{2}}{4n^{2}}g$, $\Sca^{g}=\frac{9(n-1)\gamma^{2}}{4n}$  \\
  & &    \  $-$ \  $\Ric^{c}=\frac{3(n-3)\gamma^{2}}{n^{2}}g$, $\Sca^{c}=\frac{3(n-3)\gamma^{2}}{n}$
\end{tabular}
\]
    One has to stress that   this  is not the case in general;  there exist Killing spinors with torsion (KsT) which are not   real Killing spinors, and thus manifolds which are not necessarily Einstein can be endowed with them, e.g. the Heisenberg group, see  \cite[pp.~54--57]{Julia} and \cite{AHol}.  

      We conclude that  there are  several examples of special structures   endowed with their characteristic connection which {\it fail} to carry this special kind of KsT (or TsT).  Actually,  since such a $\nabla^{c}$-parallel KsT  must be finally  a real Killing spinor, we need only to focus   on  special  structures carrying Riemannian Killing spinors. Such structures   have been classified  in dimensions $4\leq n\leq 8$   by Th. Friedrich's school in Berlin, see    \cite{Fr1980}, \cite{Fr4}, \cite{FG}, \cite{FrKath},   \cite{Gr}, \cite{FrK} and   \cite{FKMS}.  According to   \cite[Thm.~1]{FrK}, any  Einstein-Sasakian manifold   $(M^{2m+1}, g, \xi, \eta, \phi)$ admits   real Killing spinors. In particular, in dimension 5  such manifolds together wight the standard sphere exhaust all possible cases  \cite{FrKath}.  In dimension $7$,   the special structure which carries    real Killing spinors  is necessarily  a nearly parallel $\G_2$-structure, see   \cite{FKMS} for the three different types.   Finally, for higher odd dimensions $4m+1\geq 9, 4m+3\geq 11$ we know by   \cite{Bar} that only spheres, Einstein-Sasakian manifolds and 3-Sasakian manifolds can admit this special kind of spinors, while   in even dimensions, beyond the 6-dimensional nearly K\"ahler manifolds (see \cite{FG, Gr}),   the unique members are the standard spheres. We also refer to \cite[p.~143]{Ginoux} for a summary of all these results.
      
Notice now that  \cite[Lem.~2.23]{AFer} states that if   an almost contact metric structure $(M^{2m+1}, g, \xi, \eta, \phi)$ is $
 \nabla^{c}$-Einstein with respect to a characteristic connection, then it must be $\nabla^{c}$-Ricci flat. So,   such a manifold  is never strict $\nabla^{c}$-Einstein. In particular,  an Einstein-Sasaki manifold  $M^{2m+1}$   cannot be  $\nabla^{c}$-Einstein, see \cite[Rem.~2.26]{AFer},  hence compact  Einstein-Sasakian spin manifolds $(M^{2m+1}, g, \xi, \eta, \phi)$  in any odd dimension $n=2m+1\geq 5$,  although manifolds with real Killing spinors, cannot carry a $\nabla^{c}$-parallel KsT with respect to $\nabla^{s}=\nabla^{g}+2sT$.                  	The same comes true for the Tanno deformation of an Einstein-Sasakian manifold  (see below for the Tanno deformation); in the best case there is a specific parameter $t=t_{o}$  which makes  $(M^{2m+1}, g_{t}, \xi_{t}, \eta_{t}, \phi)$,  $\Ric^{\nabla^{t}}$-flat  with respect to the induced characteristic connection $\nabla^{t}=\nabla^{g_{t}}+\frac{1}{2}\eta_{t}\wedge d\eta_{t}$, see \cite[Thm.~2.24]{AFer}.
This show that the $\nabla^{c}$-Einstein condition is still very restrictive and our KsT are very special.  We remark  that  in dimension 5, and for general KsT a similar ``non-existence''  result has been described (using another integrability condition) in \cite[Cor.~A.2]{ABK}. 

After  this discussion  and due to Proposition \ref{bravo} and Theorems \ref{nk},  \ref{nG2}, \ref{s3}, we  summarise as follows: 
       \begin{theorem}
       Let $(M^{n}, g, T)$ be a compact connected Riemannian spin manifold  with $\nabla^{c}T=0$, endowed with a spinor field satisfying (\ref{cla}) with respect to the same Riemannian metric $g$.  If $n=3$, then $M^{3}\cong \Ss^{3}$ is isometric to the 3-sphere. If $n=6$, then $M^{6}$ is isometric to a strict nearly K\"ahler manifold. If $n=7$, then $M^{7}$ is isometric to a   nearly parallel $\G_2$-manifold.  
       \end{theorem}       
 It is an interesting question the existence of an  analogue of Theorem \ref{s3} for some even dimensional  sphere $\Ss^{2m}$ (different that $\Ss^{6}=\G_2/\SU_{3}$). For $\Ss^{4}$  this cannot be  the case due to \cite[Thm.~1.1]{Dalakov}.    
 
 \smallskip
 \smallskip
 \noindent{\bf A different construction.} \textnormal{  With the aim to avoid confusions, we recall that J.~Becker-Bender in her Phd thesis proved that Killing spinors with  torsion on  Einstein-Sasakian manifolds  exist \cite[Cor.~2.18]{Julia}; they appear after deforming  the metric associated to the real Killing spinors by applying the Tanno transformation, see also \cite[Ex.~5.1, 5.2]{ABK} and \cite[p.~21]{AHol}.  Let us shortly explain the {\it difference} of this construction with the present work.      For an almost contact manifold $(M^{2m+1}, g, \xi, \eta, \phi)$ the Tanno deformation \cite{Tanno}  is given by $g_{t}=tg+(t^{2}-t)\eta\otimes\eta$,  $\xi_{t}=\frac{1}{t}\xi $ and $\eta_{t}=t\eta$ for some $t>0$ (for details on Sasakian geometry  see \cite{Tanno, FrK, FKMS, Boy1, FrIv, Srni}).  If $(M^{2m+1}, g, \xi, \eta, \phi)$ is Sasakian, then $(M^{2m+1}, g_{t}, \xi_{t}, \eta_{t}, \phi)$ is too \cite{Tanno, FrKim}. Consider   the Tanno deformation of an Einstein-Sasakian manifold $(M^{2m+1}, g, \xi, \eta, \phi)$ with $2m+1\geq 5$.  In \cite[Thm.~2.22]{Julia} it was shown that  $(M^{2m+1}, g_{t}, \xi_{t}, \eta_{t}, \phi)$ admits  Killing spinors with torsion for the parameters $s_{t}=\frac{m+1}{4(m-1)}(\frac{1}{t}-1)$, with respect to the connection
\[
\nabla^{s_{t}}_{X}\varphi=\nabla^{g_{t}}_{X}\varphi+s_{t}(X\lrcorner T^{c})\cdot\varphi, \quad T^{c}:=\eta_{t}\wedge d\eta_{t}=2\eta_{t}\wedge F_{t},
\]
with Killing numbers $\zeta_{1, t}=\frac{\varepsilon}{2}(1-4s_{t})$ and  $\zeta_{2, t}=(-1)^{m+1}\zeta_{1, t}$,
 respectively.
 Here,   
 $\varepsilon=\pm 1$ is the number defined by the equation $e_{1}\cdot\phi(e_{1})\cdot\ldots\cdot e_{m}\cdot\phi(e_{m})\cdot\xi\cdot \varphi=\varepsilon i^{m+1}\varphi$ for a local orthonormal frame $\{e_{1}, \phi(e_{1}), \ldots, e_{m}, \phi(e_{m}), \xi\}$ of $M^{2m+1}$. If there is no deformation  $(t=1)$, then $s_{t}=0$ and $\varphi_{i}$  coincide with  the Riemannian Killng spinors that $(M^{2m+1}, g, \xi, \eta, \phi)$ carries, see  \cite[Thm.~1]{FrK}.  For $1-4s_{t}=0$, i.e. the Riemannian metric $g_{t_{o}}=g_{{m+1}/{2m}}$,  the spinor fields are $\nabla^{c}$-parallel,  see also \cite[p.~21]{AHol}. However, $\varphi_{i}$ are KsT for any Riemannian metric $g_{t}$ with $t>0, t\neq \frac{m+1}{2m}$.  This shows that the equations given in (\ref{cla}) hold with respect to different  Riemannian metrics. Maybe it is an interesting but difficult  task to describe new non-integrable $G$-structures  carrying   Killing spinors with torsion   after a compatible deformation of the given $G$-structure, if any.}

    \section{Further applications}\label{finals} 
     \subsection{A converse direction of Proposition \ref{bravo}} 
 In the proof of Proposition \ref{bravo} we proved  that a $\nabla^{c}$-parallel twistor spinor with torsion $\varphi\in\ke(P^{s}\big|_{\Sigma_{\gamma}})$ for some $\gamma\neq 0$ and $s\neq 1/4$, satisfies the equation
    \begin{equation}\label{check1}
     \sum_{i}T(X, e_{i})\cdot (e_{i}\lrcorner T)\cdot\varphi=-\frac{18\gamma^{2}}{n^{2}}X\cdot\varphi, \quad\forall  X\in\Gamma(TM). 
    \end{equation}
Next we will show that   (\ref{check1}) is still true if: 
 \begin{lemma}\label{observe}
   Let $(M^{n}, g, T)$ $(n>3)$ be a (compact) Riemannian manifold with $\nabla^{c}T=0$, carrying a    $\nabla^{c}$-parallel spinor field $0\neq\varphi\in\Sigma_{\gamma}$ for   some $0\neq \gamma\in\Spec(T)$, where $\nabla^{c}=\nabla^{g}+\frac{1}{2}T$  is the characteristic connection.
    Assume that $M^{n}$ is both Einstein and $\nabla^{c}$-Einstein with respect to $g$, in particular that  the relations (\ref{rig}) and (\ref{ricapl}) are satisfied for any  $X\in\Gamma(TM)$.   Then, any vector field  $X\in\Gamma(TM)$ satisfies   (\ref{check1}), as well. 
      \end{lemma}
 \begin{proof}
Let $\{e_{1}, \ldots, e_{n}\}$ be an orthonormal frame of $M^{n}$ with respect to $g$. Under our assumptions the following relations are true for some $\gamma\neq 0$:  $\sum_{i}e_{i}\cdot R^{g}(X, e_{i})\varphi=-\frac{1}{2}\Ric^{g}(X)\cdot\varphi=-\frac{9(n-1)\gamma^{2}}{8n^{2}}X\cdot\varphi$ and   $(X\lrcorner \sigma_{T})\cdot\varphi= \Ric^{c}(X)\cdot\varphi =\frac{3(n-3)\gamma^{2}}{n^{2}}X\cdot\varphi$. Then, equation  (\ref{check1}) it is a simple consequence of    (\ref{conv}) and the $\nabla^{c}$-parallelism of $\varphi$. %
 \end{proof}
\textnormal{Under the assumptions of Lemma \ref{observe} and for general $n$, we are not able to show  that equation (\ref{check1}) implies the twistor equation.  However, this is possible  for $n=6$ and a nearly K\"ahler manifold, or $n=7$ and a  nearly parallel $\G_2$-manifold, in combination with the relations  (\ref{wrong}) and (\ref{xtg2}), respectively.  Notice that although in our text  these formulas  appear as a consequence of  the twistor equation with torsion, both are deeper consequences of the  $\nabla^{c}$-parallelism  of the spinors $\varphi^{\pm}$  (respectively $\varphi_{0}$) and the character of the spin representation  in these dimensions, see for example \cite[Lem.~2.2, 2.3]{ACFH}. Further applications of the relations  (\ref{wrong}) and (\ref{xtg2}) for these two kinds of weak holonomy structures, will be shortly described in an appendix.}

\smallskip
  In \cite{ABK}, a full integrability condition for the existence of Killing spinors with torsion with respect to the family $\nabla^{s}$ for $s=\frac{n-1}{4(n-3)}$ was presented. For convenience, let us  recall the statement.
\begin{theorem} \textnormal{(\cite[Thm.~A.2, p.~324]{ABK})}\label{integrab} 
Let  $\varphi$ be a KsT with respect to $\nabla^{s}$ for $s=\frac{n-1}{4(n-3)}$ with Killing number $\zeta$. Set $\lambda:=1/2(n-3)$. Then 
\begin{eqnarray}
\Ric^{c}(X)\cdot\varphi&=&-16s\zeta(X\lrcorner T)\cdot\varphi+4(n-1)\zeta^{2}X\cdot\varphi+(1-12\lambda^{2})(X\lrcorner \sigma_{T})\cdot\varphi\nonumber\\
&&-2(2\lambda^{2}+\lambda)  \sum_{i}T(X, e_{i})\cdot (e_{i}\lrcorner T)\cdot\varphi. \label{mala}
\end{eqnarray}
\end{theorem}

For a Riemannian  manifold $(M^{n}, g, T)$ with $\nabla^{c}T=0$ one  can use  Theorem \ref{general1} and Proposition \ref{bravo}  to prove that if $\varphi\in\cal{K}^{s}(M, g)_{\zeta}$ is a KsT for $s=\frac{n-1}{4(n-3)}\neq 0, 1/4$  with Killing number  $\zeta=3\gamma(1-4s)/4n$, which  is in addition   $\nabla^{c}$-parallel  w.r.t.  the characteristic connection $\nabla^{c}=\nabla^{g}+\frac{1}{2}T$,   then (\ref{mala}) reduces to an identity, namely the twistor equation with torsion. In fact, for such $\nabla^{c}$-parallel KsT w.r.t.  $\nabla^{s}=\nabla^{g}+2sT$, the optimal integrability conditions are these described by   Theorem \ref{bravo}.
 \begin{corol}\label{little}
 Let $(M^{n}, g, T)$  $(n>3)$ be a compact connected Riemannian manifold  with $\nabla^{c}T=0$,  where $\nabla^{c}=\nabla^{g}+\frac{1}{2}T$ is the characteristic connection. Consider a $\nabla^{c}$-parallel Killing spinor with torsion  $\varphi_{0}\in\cal{K}^{s}(M^{n}, g)_{\zeta}$ for  $s=\frac{n-1}{4(n-3)}\neq 0, 1/4$,  with Killing number $\zeta=3\gamma(1-4s)/4n$ for some $0\neq \gamma\in\Spec(T)$.  Then,   $\varphi_{0}$ satisfies the condition  (\ref{mala})  as an identity.
 \end{corol}
 \begin{proof}
  Consider a spinor field $\varphi_{0}\in\Gamma(\Sigma)$ which  is $\nabla^{c}$-parallel and the same time a Killing spinor with torsion for $s=\frac{n-1}{4(n-3)}\neq 0, 1/4$. By assumption,  the Killing number $\zeta$ is given by $\zeta=\frac{3(1-4s)\gamma}{4n}$, i.e. $\zeta=-\frac{3\gamma}{2n(n-3)}$ for some $T$-eigenvalue $\gamma\neq 0$.  Automatically, $\varphi$ is an element in the kernel of the twistor operator $P^{s}$ for some $s\neq 1/4$, thus  according to Proposition \ref{bravo}, the relations  (\ref{rig}) and (\ref{ricapl}) are satisfied for any  $X\in\Gamma(TM)$.  Moreover, by  the proof of the same proposition (or  Lemma \ref{observe}, since the condition $\varphi\in\Sigma_{\gamma}$ follows by the twistor equation)   we   know  that also relation (\ref{check1}) needs to hold.
Combining   this information, is  a straightforward    computation to show that the integrability condition given by (\ref{mala})    reduces to $(X\lrcorner T)\cdot\varphi_{0}+\frac{3\gamma}{n}X\cdot\varphi_{0}=0$,  $\forall X\in\Gamma(TM)$.  \end{proof}

\subsection{The relation between the two estimates and the role of Proposition \ref{klik}}    
          For the first eigenvalue of  $\slashed{D}^{2}$ restricted on $\Sigma_{\gamma}$ there is a second estimate, the so-called {\it twistorial estimate}   introduced in  \cite{ABK, Julia}. This is  defined by \begin{equation}\label{twestim}
\lambda_{1}(\slashed{D}^{2}|_{\Sigma_{\gamma}})\geq \frac{n}{4(n-1)}\Sca^{g}_{\rm min}+\frac{n(n-5)}{8(n-3)^{2}}\|T\|^{2}+\frac{n(4-n)}{4(n-3)^{2}}\gamma^{2}:=\be_{\rm tw}(\gamma),
\end{equation}
and the equality case appears  if and only if $\varphi$ is a twistor spinor with torsion for $s=(n-1)/4(n-3)$ and $\Sca^{g}=\text{constant}$.  A similar expression with (\ref{twestim}) holds for the whole spinor bundle.
In  the presence of a $\nabla^{c}$-parallel spinor  $0\neq \varphi\in\Sigma_{\gamma}$, the inequality $\be_{\rm tw}(\gamma)\leq\be_{\rm univ}(\gamma)$ needs to holds.  For  $n\leq 8$, this fact implies the inequalities \cite[Lem.~4.1]{ABK}
\begin{equation}\label{inqilkas}
0\leq 2n\|T\|^{2}+(n-9)\gamma^{2}, \quad\quad \Sca^{g}\leq \frac{9(n-1)}{2(9-n)}\|T\|^{2}.
\end{equation}
 The equality case here, takes place if and only if  the universal   estimate   coincides with the twistorial estimate. However, if one of these inequalities holds as an equality, then  Proposition  \ref{klik}  states that $\varphi$  must be  a real Killing spinor. In fact, one can  say more:     
\begin{prop}\label{afterklik}
  Let $(M^{n}, g, \nabla^{c})$ $(3< n\leq 8)$ be a compact  connected Riemannian spin manifold with $\nabla^{c}T=0$ and positive scalar curvature, carrying a spinor field $\varphi\in\Gamma(\Sigma)$   satisfying the equations given by (\ref{fund1}) for some $0\neq\gamma\in\Spec(T)$.
   If  $\be_{\rm tw}(\gamma)=\be_{\rm univ}(\gamma)$ then $\varphi$ is a real Killing spinor with respect to $g$, with Killing number $\kappa=3\gamma/4n$.  Conversely, if  $\varphi$ is a real Killing spinor with $\kappa=3\gamma/4n$ satisfying $(\ref{fund1})$, then     $\be_{\rm tw}(\gamma)=\be_{\rm univ}(\gamma)$ identically. 
\end{prop}
\begin{proof}
Since  $\be_{\rm tw}(\gamma)=\be_{\rm univ}(\gamma)$ if and only if one of the above inequalities holds as equality, the one  direction is obvious due to  Proposition  \ref{klik}, which is their equality case.  In fact,  under our assumptions,  the scalar curvature $\Sca^{g}$ is constant (see Theorem  \ref{KNOWN1}) and  the universal estimate becomes sharp.  Since $\slashed{D}=D^{c}-\frac{1}{2}T$, the spinor $\varphi$ is an eigenspinor of the cubic Dirac operator    with eigenvalue $-\frac{1}{2}\gamma$; therefore $\lambda_{1}=\frac{1}{4}\gamma^{2}=\be_{\rm univ}(\gamma)$. Since the two estimates coincide, we conclude that  $\be_{\rm tw}(\gamma)=\frac{1}{4}\gamma^{2}$ also,   where    $\be_{\rm tw}(\gamma)$ is given now  by the equality case in (\ref{twestim}) and without the minimal condition in $\Sca^{g}$.  Hence, for this direction one can  also  present another proof based on Theorem \ref{general1}.  According to \cite{ABK}, if the two estimates coincide (and   both hold with the equality),  then the   spinor $\varphi\in\Sigma_{\gamma}$ must be $\nabla^{c}$-parallel and also a   twistor spinor with torsion for $s=(n-1)/4(n-3)$.  Then, by Theorem \ref{general1},  $\varphi$  is also a Killing spinor with torsion   with Killling number $\zeta=\frac{3(1-4s)\gamma}{4n}=-\frac{3\gamma}{2n(n-3)}$ and moreover a  real Killing spinor with Killing number $\kappa=3\gamma/4n$. 

We  present  a  proof for the converse direction. If $\varphi$ is a real Killing spinor on $(M^{n}, g,  T)$ with $\kappa=3\gamma/4n$, then $\Sca^{g}=\frac{9(n-1)}{4n}\gamma^{2}$ is constant; because $\varphi\in\Sigma_{\gamma}$ is $\nabla^{c}$-parallel and   the scalar curvature satisfies the desired formula, by Proposition \ref{klik} for $n\leq 8$ we also have $\|T\|^{2}=\frac{2(9-n)}{9(n-1)}\Sca^{g}$.  Moreover, Theorem \ref{general1} tell us that this is also a twistor spinor for some $s\neq  1/4$; thus one may assume without loss of generality  that  $s=(n-1)/4(n-3)\neq 1/4$. Hence, finally   the twistorial estimate must hold as an equality and  replacing the previous values,   we obtain the expression
\[
\be_{\rm tw}(\gamma)=\frac{n\Big[9(n-3)^{2}+(n-5)(9-n)+4n(4-n)\Big]}{36(n-1)(n-3)^{2}}\Sca^{g}.
\]
Equivalent expressions in terms only of $\|T\|^{2}$ or   $\gamma^{2}$, can be easily deduced. The same time,   $\varphi\in\Sigma_{\gamma}$ is $\nabla^{c}$-parallel, hence  for the universal estimate we get $\be_{\rm univ}(\gamma)=\frac{\gamma^{2}}{4}=\frac{n}{9(n-1)}\Sca^{g}$. Then,  one can easily check that the equation $\be_{\rm tw}(\gamma)-\be_{\rm univ}(\gamma)=0$  holds as an identity.  \end{proof}
\begin{example}
\textnormal{For $n=7$ and for a nearly parallel $\G_2$-manifold $(M^{7}, g, \omega)$ we compute  $\be_{\rm tw}(\gamma)=\frac{7}{54}\Sca^{g}=\be_{\rm univ}(\gamma)$.  For $n=6$ and a nearly K\"ahler manifold $(M^{6}, g, J)$ we have $\be_{\rm tw}(\gamma)=\frac{2}{15}\Sca^{g}=\be_{\rm univ}(\gamma)$, see also \cite[Ex.~6.1]{ABK}. }  
\end{example}

 \appendix
 \section{The endomorphism $\sigma_{T}$ on the spinor bundle $\Sigma$}
 
 Relations (\ref{wrong}) and (\ref{xtg2}) have another important consequence,  related with the eigenspinors   of the endomorphism defined by the 4-form $
\sigma_{T}$ (or equivalently $dT$) in the Clifford algebra.   In particular, the results that  we describe below  are known (see \cite[Lem.~10.7]{FrIv}, \cite[Thm.~1.1]{AFNP}),  but our proofs are different.

\begin{prop}\textnormal{(\cite[Lem.~10.7]{FrIv}, \cite[Thm.~1.1]{AFNP})} \label{stat}
On a nearly K\"ahler manifold $(M^{6}, g, J)$ the spinor fields $\varphi^{\pm}\in\Sigma_{\pm 2\|T\|}$ are eigenspinors of the endomorphism defined by the 4-form $
\sigma_{T}$. In particular,
 \[
 \sigma_{T}\cdot\varphi^{\pm} =-\frac{\Sca^{c}}{4}\cdot\varphi^{\pm}=-\frac{3}{2}\|T\|^{2}\cdot\varphi^{\pm}=-3\tau_{0}\cdot\varphi^{\pm}.
\]
Similarly,  on a  proper nearly parallel $\G_2$-manifold  the spinor field  $\varphi_{0}\in\Sigma_{-\sqrt{7}\|T\|}$ is eigenspinor  of the endomorphism defined by the 4-form $
\sigma_{T}$. In particular,
\[
\sigma_{T}\cdot\varphi_{0} =-\frac{\Sca^{c}}{4}\cdot\varphi_{0}=-3\|T\|^{2}\cdot\varphi_{0}=-\frac{7}{12}\tau_{0}^{2}\cdot\varphi_{0}.
\]
\end{prop}
 \begin{proof}
  We present a direct proof based on (\ref{wrong}) and (\ref{xtg2}), respectively.  Consider a local orthonormal frame $\{e_{i}\}$.  In the nearly K\"ahler case $(M^{6}, g, J)$ and due to (\ref{wrong}), it holds that
 \[
 (e_{i}\lrcorner T)\cdot \varphi^{\pm}=\mp \|T\|e_{i}\cdot \varphi^{\pm}, \quad \forall i\in\{1, \ldots, 6\}.
  \]
 Then, because $(2\sigma_{T}-3\|T\|^{2})\cdot\varphi^{\pm}=\sum_{i}(e_{i}\lrcorner T)\cdot (e_{i}\lrcorner T)\cdot\varphi^{\pm}$  (see Appendix C in \cite{ABK}), we immediately get
 \begin{eqnarray*}
 (2\sigma_{T}-3\|T\|^{2})\cdot\varphi^{\pm} =\mp\|T\|\sum_{i}(e_{i}\lrcorner T)\cdot e_{i}\cdot \varphi^{\pm}  &=&\mp 3\|T\| T\cdot\varphi^{\pm}\\
 &=&\mp 3\|T\| \Big(\pm 2\|T\|\Big)\cdot\varphi^{\pm}=-6\|T\|^{2}\cdot\varphi^{\pm}.
  \end{eqnarray*}
  Similarly, for a 7-dimensional  (proper) nearly parallel $\G_2$-manifold $(M^{7}, g, \omega)$ we easily conclude that
   \begin{eqnarray*}
 (2\sigma_{T}-3\|T\|^{2})\cdot\varphi_{0}&=&\sum_{i}(e_{i}\lrcorner T)\cdot (e_{i}\lrcorner T)\cdot\varphi_{0}\overset{(\ref{xtg2})}{=}\frac{3\|T\|}{\sqrt{7}}\sum_{i}(e_{i}\lrcorner T)\cdot e_{i}\cdot \varphi_{0}\\
 &=&\frac{9\|T\|}{\sqrt{7}}  T\cdot\varphi_{0}\overset{(\ref{teigen})}{=} -9 \|T\|^{2}\cdot\varphi_{0},
  \end{eqnarray*}
  and the claim follows. 
   \end{proof}
 
 Combining the expression for   $\sigma_{T}$ and the equality $T^{2}=-2\sigma_{T}+\|T\|^{2}$, it is easy to compute  also the action of $T^{2}$ on  $\nabla^{c}$-parallel spinors lying in $\Sigma_{\gamma}$.  In particular,  for a $\nabla^{c}$-parallel spinor $\varphi$ the action $T^{2}\cdot\varphi$ in encrypted  in the kernel of the Casimir operator $\Omega:=\Delta_{T}+\frac{1}{16}\Big[2\Sca^{g}+\|T\|^{2}\Big]-\frac{1}{4}T^{2}$ \cite{Cas}.  Any $\nabla^{c}$-parallel belongs in the kernel of $\Omega$, hence it  satisfies the equation (see for example \cite{3Sak})
  \[
  T^{2}\cdot\varphi=\frac{1}{4}\Big[2\Sca^{g}+\|T\|^{2}\Big]\cdot\varphi. 
  \]
This formula in combination with $T^{2}=-2\sigma_{T}+\|T\|^{2}$,  gives rise to another way  for the computation of  $\sigma_{T}\cdot\varphi$. Finally, for the action $T^{2}\cdot\varphi$ the relation $T\cdot\varphi=\gamma\varphi$ can be applied twice which yields the same result, i.e. $T^{2}\cdot\varphi=\gamma^{2}\varphi$ with $\gamma^{2}=\frac{1}{4}\Big[2\Sca^{g}+\|T\|^{2}\Big]$ by Theorem \ref{KNOWN1}. 

 \end{document}